\documentclass[11pt]{amsart}
\title{Comparing Periodic Point Invariants for Parameterized Families of Maps}
\author{Lucas Williams}
\address{Department of Mathematics, Purdue University}
\email{will4135@purdue.edu}

\usepackage{amsmath,amssymb,amsthm,array,color,multirow,pdflscape,mathrsfs,subcaption,todonotes}
\usepackage{tikz,tikz-cd}
\usetikzlibrary{bending}
\usepackage{float}
\usepackage{imakeidx}
\usepackage[font=normalsize]{caption}
\usepackage[margin=1.25in]{geometry}
\makeatletter
\def\l@section{\@tocline{1}{0pt}{1pc}{}{}}
\def\l@subsection{\@tocline{2}{0pt}{1pc}{4.6em}{}}
\def\l@subsubsection{\@tocline{3}{0pt}{1pc}{7.6em}{}}
\renewcommand{\tocsection}[3]{%
	\indentlabel{\@ifnotempty{#2}{\makebox[2.3em][l]{%
				\ignorespaces#1 #2.\hfill}}}#3}
\renewcommand{\tocsubsection}[3]{%
	\indentlabel{\@ifnotempty{#2}{\hspace*{2.3em}\makebox[2.3em][l]{%
				\ignorespaces#1 #2.\hfill}}}#3}
\renewcommand{\tocsubsubsection}[3]{%
	\indentlabel{\@ifnotempty{#2}{\hspace*{4.6em}\makebox[3em][l]{%
				\ignorespaces#1 #2.\hfill}}}#3}
\makeatother

\newcommand{\R}{\mathbb R}

\newcommand{\Z}{\mathbb{Z}}

\newcommand{\Sph}{\mathbb{S}}
\newcommand{\RP}{\mathbb{RP}}

\newcommand{\id}{\textup{id}}

\newcommand{\colim}{\textup{colim}\,}

\newcommand{\holim}{\textup{holim}\,}

\newcommand{\sma}{\wedge}
\newcommand{\barsmash}{\,\overline\wedge\,}

\newcommand{\tr}{\textup{tr}\,}
\newcommand{\THH}{\textup{THH}}

\newcommand{\TR}{\textup{TR}}

\newcommand{\po}[2]{\ar@{}@<{#2}>[rd]|({#1})*\txt{\Large $\ulcorner$}}
\newcommand{\pb}[2]{\ar@{}@<{#2}>[rd]|({#1})*\txt{\Large $\lrcorner$}}

\DeclareMathOperator{\std}{std}
\usepackage[unicode=true, pdfusetitle,
 bookmarks=true,bookmarksnumbered=false,
 breaklinks=false,
 backref=false,
 colorlinks=true,
 linkcolor=blue,
 citecolor=blue,
 urlcolor=blue,
 final
]{hyperref}
\usepackage{amsthm}
\usepackage{aliascnt}
\usepackage{import}
\usepackage[capitalise]{cleveref}

\newcommand{\newrefformat}[2]{}

\newaliascnt{dia}{equation}
\aliascntresetthe{dia}



\crefname{dia}{Diagram}{Diagrams}
\creflabelformat{dia}{#2\textup{(#1)}#3}

\makeatletter
\newenvironment{diagram*}[1][]{%
    \begin{equation*}%
    \begin{tikzcd}[#1]%
    \refstepcounter{dia}
}{%
    \end{tikzcd}%
    \end{equation*}%
}
\makeatother

\crefname{dia}{Diagram}{Diagrams}
\creflabelformat{dia}{#2\textup{(#1)}#3}

\theoremstyle{plain}   
\newtheorem{thm}{Theorem}[section] 
\makeatletter\let\c@thm\c@thm\makeatother
\newtheorem{cor}[thm]{Corollary}
\makeatletter\let\c@cor\c@thm\makeatother
\newtheorem{lem}[thm]{Lemma}
\makeatletter\let\c@lemma\c@thm\makeatother
\newtheorem{prop}[thm]{Proposition}
\makeatletter\let\c@prop\c@thm\makeatother

\makeatletter\let\c@claim\c@thm\makeatother

\theoremstyle{definition}

\newtheorem{df}[thm]{Definition}
\makeatletter\let\c@defn\c@thm\makeatother
\newtheorem{const}[thm]{Construction}
\makeatletter\let\c@const\c@thm\makeatother
\newtheorem{notn}[thm]{Notation}
\makeatletter\let\c@notn\c@thm\makeatother

\makeatletter\let\c@outline\c@thm\makeatother

\makeatletter\let\c@propty\c@thm\makeatother

\makeatletter\let\c@problem\c@thm\makeatother
\newtheorem{conj}[thm]{Conjecture}
\makeatletter\let\c@conj\c@thm\makeatother

\theoremstyle{remark}

\newtheorem{rmk}[thm]{Remark}
\makeatletter\let\c@rem\c@thm\makeatother
\newtheorem{ex}[thm]{Example}
\makeatletter\let\c@ex\c@thm\makeatother

\makeatletter\let\c@observationn\c@thm\makeatother

\makeatletter
\let\c@equation\c@thm
\numberwithin{equation}{section}
\makeatother

\crefname{lemma}{Lemma}{Lemmas}
\crefname{thm}{Theorem}{Theorems}
\crefname{defn}{Definition}{Definitions}
\crefname{notn}{Notation}{Notations}
\crefname{const}{Construction}{Constructions}
\crefname{prop}{Proposition}{Propositions}
\crefname{rem}{Remark}{Remarks}
\crefname{cor}{Corollary}{Corollaries}
\crefname{equation}{Equation}{Equations}
\crefname{ex}{Example}{Examples}
\crefname{propty}{Property}{Properties}
\crefname{problem}{Problem}{Problems}

\hypersetup{linktoc=all}
\begin{document}

\begin{abstract}
	We compare different periodic point invariants for families of maps parameterized over a compact manifold. Malkiewich and Ponto showed that, in the case of a single map, the Fuller trace is equivalent to the collection of Reidmeister traces of iterates. In this paper, we show that, in contrast to the case of a single map, the fiberwise Fuller trace is a strictly more sensitive invariant than the collection of fiberwise Reidemeister traces of iterates. This resolves a conjecture of Malkiewich and Ponto. We also prove that our construction gives a geometric representation of elements of the kernel of the ghost map out of topological restriction homology. 
\end{abstract}

\maketitle

\setcounter{tocdepth}{2}
\tableofcontents

\parskip 2ex

\section{Introduction}

Let $M$ be a compact manifold and $f\colon M\to M$ a continuous map. The Lefschetz trace of $f$, denoted $L(f)\in \Z$, can be interpreted as a weighted sum of the fixed points of $f$. The Lefschetz trace has many generalizations which have been studied from a homotopical and $K$-theoretic vantage point in papers such as \cite{dold_puppe, klein_williams, mp2, ponto_asterisque, gn_ktheory}. In the present work, we will focus on generalizations of the Lefschetz trace that count the $n$-periodic points of $f$, or equivalently the fixed points of $f^n$, where $f\colon E\to E$ is a family of maps parameterized over a compact manifold. 

In the case of a single map, $f\colon M\to M$, the Lefschetz traces vanishes if $f$ is homotopic to a map with no fixed points. The converse however, is false. To obtain a complete obstruction to the removal of fixed points from $f$, the Lefschetz trace must be refined to the Reidemeister trace $R(f)$.  The Reidemeister trace takes values in the 0th homology of the twisted free loop space of $f$, which is defined as
\[
\Lambda^fM= \{\gamma\colon I\to M\mid f(\gamma(1)) = \gamma(0)\}.
\]
If $M$ is a compact manifold of dimension not equal to 2, then the Reidemeister trace, $R(f)$, is a complete obstruction to the removal of fixed points from $f$. In other words, $f$ is homotopic to a map without fixed points if and only if $R(f)$ vanishes \cite{jiang80, wecken}.

In \cite{mp2}, Malkiewich and Ponto study several invariants for $n$-periodic points of $f\colon M\to M$, or rather, the fixed points of $f^n$. The easiest of these to define is the collection of Reidemeister traces of iterates, $\{R(f^k): k\mid n\}$. The other invariant is a $C_n$-equivariant construction called the Fuller trace of $f$. The Fuller trace is defined by first observing that the fixed points of the $C_n$-equivariant map
\begin{align*}
M\times\dots\times M & \xrightarrow{\Psi^n(f)} M\times \dots \times M\\
(x_1,\dots,x_n) & \mapsto (f(x_n),f(x_1),\dots,f(x_{n-1}))
\end{align*}
are precisely the $n$-periodic points of $f\colon M\to M$ \cite{fuller1967index, dold83}. Here the $C_n$-action cycles the coordinates of $M^{\times n}$. Moreover, if $f$ is homotopic to a map with no $n$-periodic points then $\Psi^n(f)$ is $C_n$-equivariantly homotopic to a map with no fixed points. The $n$th Fuller trace of $f$ is the $C_n$-equivariant Reidemeister trace of $\Psi^n(f)$, and is an obstruction to removing the $n$-periodic points from $f$ by a homotopy. The $n$-Fuller trace is a map of genuine $C_n$-spectra
\[
R(\Psi^n(f))^{C_n}\colon \Sph\to \Sigma_+^\infty \Lambda^{\Psi^nf}M^{\times n}.
\]

In \cite[Corollary 1.3]{mp2}, Malkiewich and Ponto answer a conjecture of \cite{klein_williams_2} by proving that the set of Reidemeister traces $\{R(f^k): k\mid n\}$ can be recovered from the Fuller trace $R(\Psi^n(f))^{C_n}$. The vanishing of $R(\Psi^n(f))^{C_n}$ implies the vanishing of $R(f^k)$ for all $k|n$. They do this by interpreting Reidemeister traces and Fuller traces as classes in $\pi_0(-0)$ of topological Hoschild homology and topological restriction homology respectively. They then leverage the fact that the ghost map between these spectra is injective on $\pi_0$. 

Combining this result with \cite{jezierski_cancelling} yields the following Corollary. If $M$ is a smooth (or PL) compact manifold of dimension at least three, $R(f^k)$ vanishes in the homotopy category of spectra for all $k|n$ if and only if $R(\Psi^n(f))^{C_n}$ vanishes in the homotopy category of $C_n$-spectra if and only if $f$ is homotopic to a map with no $n$-periodic points. 

To summarize, for sufficiently high dimensional manifolds, the Fuller trace and the collection of Reidemeister traces of iterates are equivalent and both are complete obstructions to the removal of $n$-periodic points. 

The present work seeks to study the generalizations of these invariants to the fiberwise setting. In contrast to the setting of a single map, we will show that the fiberwise Fuller trace is not equivalent to the collection of fiberwise Reidemeister traces of iterates. To be precise, we prove the following theorem which resolves \cite[Conjecture 1.9]{mp2}. 

\begin{thm}\label{thm:main}
There is a family of maps $f\colon E\to E$ parameterized over a base space $B$ for which the fiberwise Reidemeister traces $R_B(f)$ and $R_B(f^2)$ both vanish, but for which the fiberwise Fuller trace $R_{B}(\Psi_B^2 f)^{C_n}$ is non-vanishing.
\end{thm}

This theorem demonstrates that, in contrast to the case of a single map, the fiberwise Fuller trace is a strictly stronger invariant than the collection of fiberwise Reidemeister traces of iterates. 

After completing the various computations that constitute the proof of \cref{thm:main}. We will discuss the relationship between \cref{thm:main} and trace methods in $K$-theory. In particular, we will prove that our main construction of a bundle endomorphism gives rise to an element of $\pi_1$ of the kernel of the ghost map out of topological restriction homology of the sphere. 

In fact, analysis of the ghost map $\TR(\Sph)\to\prod_{n\geq 1}\THH(\Sph)$ gave rise to \cite[Conjecture 1.9]{mp2} which we resolve in this paper. In particular, the so called ``infinite Fuller trace'' $R(\Psi^\infty(f))^{C_\infty}$ defines a class in $\pi_0\left(TR(\Sph)\right)$ while the collection of Reidemeister traces $R(f^n)$ give a class in $\prod_{n\geq 1} \pi_0\Sph$. The ghost map $\TR(\Sph)\to\prod_{n\geq 1}\Sph$ sends $R(\Psi^\infty(f))^{C_\infty}$ to $\prod_{n\geq 1} R(f^n)$ on $\pi_0$. The ghost map is injective on $\pi_0$ where $\TR$ recovers the ring of Witt vectors \cite{bhm,madsen_survey}. This injectivity is precisely what allows Malkiewich and Ponto to conclude that the Fuller trace $R(\Psi^n(f))^{C_n}$ vanishes if and only if the set of Reidemeister traces $\{R(f^k): k\mid n \}$ are all zero. 

However, the ghost map fails to be injective on higher homotopy groups. We can produce classes in these higher homotopy groups of $TR(\Sph)$ by considering periodic point invariants of manifolds bundles. In the final section of this paper we will prove that parameterized periodic point invariants considered earlier give rise to such a class in $\pi_1(\TR(\Sph))$. Moreover, we will prove that our construction gives rise to an element of the kernel of the ghost map on $\pi_1$. In \cite{CLMPZ} the authors give an interpretation of $\pi_0(\TR(-))$ as housing periodic point invariants. They summarize their results with the slogan:

``\textit{$\THH$ is the natural home for fixed point invariants and $\TR$ is the natural home for periodic point invariants.}"

This interpretation places an emphasis on spectra like $\TR$ and $\THH$ as interesting object of study for dynamical reasons by housing traces rather than simply as stepping stones in $K$-theory computations. In this paper we build on that idea by focusing on the ghost map $\TR\to \prod \THH$ from the view of periodic points. We add to the above slogan with the following:

\textit{The kernel of the ghost map $\TR \to \prod\THH$ measures the failure of fixed point invariants to detect periodic points}

In future work, joint with Cary Malkiewich and Kate Ponto, we plan to follow up the present paper by investigating the following conjecture.

\begin{conj}\cite[Conjecture 1.8]{mp2}
The fiberwise Fuller trace, $R_{B}(\Psi_B^n f)^{C_n}$, is a complete obstruction to the removal of $n$-periodic points from a family of endomorphisms $f\colon E\to E$ over $B$ when $B$ is a finite dimensional cell complex and $E\to B$ is a manifold bundle with fibers of dimension $3+\dim(B)$. 
\end{conj}

\subsection{Organization}

In section two, we provide background on invariants of fixed and periodic points. In section three, we give an explicit construction which realizes \cref{thm:main}. Sections four is devoted to proving \cref{thm:main}. Finally, in section five, we prove that the main construction of this paper gives rise to an element of the kernel of the ghost map out of topological restriction homology on the first stable homotopy group. 

\subsection{Acknowledgments}

Many thanks to John Klein, Cary Malkiewich, and Kate Ponto for helpful conversations and suggestions. Thanks also to the organizing and scientific committees of the 2022 workshop ``Homotopical Methods in Fixed Point Theory" which served as the author's introduction to this subject matter. This paper represents a part of the author's dissertation written under the supervision of Cary Malkiewich at Binghamton University. 

\section{Preliminaries}

In this section we give the necessary background for the paper. We will begin with a minimal discussion on equivariant parameterized stable homotopy theory. We will then cover the relevant background in fixed point theory. Finally, since our main computations will be cobordism theoretic, we will give a few preliminaries on this topic as well.

\begin{df}
Let $B$ be a $G$-space. A \textbf{retractive $G$-space over $B$} is a $G$-space $X$ equipped with an equivariant retract $p\colon X\to B$. 
\end{df}

\begin{df}\label{df:external smash}
Let $B$ be a $G$-space and $X$ a retractive $G$-space over $B$. Let $A$ be a $G$-space which we will think of as a retractive $G$-space over $*$. The \textbf{external smash product of $A$ and $X$}, denoted $A\barsmash X$, is a retractive $G$-space over $B$ defined as the pushout:
\begin{figure}[H]
\center
\begin{tikzcd}
A\times B\cup_{*\times B} *\times X \arrow[r]\arrow[d] & A\times X\arrow[d]\\
*\times B \arrow[r]& \arrow[ul, phantom, "\ulcorner", very near start]A\barsmash  X.
\end{tikzcd}
\end{figure}
\end{df}

\begin{df}\label{df:suspension}
Let $B$ be a $G$-space and $X$ a retractive $G$-space over $B$. The \textbf{suspension of $X$ over $B$} is $\Sigma_B X=S^1\barsmash X$. Note that the $k$-fold suspension of $X$ over $B$ is equivariantly homeomorphic to $S^k\barsmash X$. 
\end{df}

\begin{df}
Let $B$ be a $G$-space which we regard as a $(G\times O(n))$-space by letting $O(n)$ act trivially. A \textbf{parameterized orthogonal $G$-spectrum over $B$} is a sequence of retractive $(G\times O(n))$-spaces over $B$, denoted $X_n$, equipped with $G$-equivariant structure maps $\Sigma_BX_n\to X_{n+1}$ such that the composite 
\[
\Sigma_B^p X_q \to \dots \to X_{p+q}
\]
is $(O(p)\times O(q))$-equivariant. 
\end{df}

\begin{rmk}
In the previous definition, when $B$ is a single point, we recover the definition of an orthogonal $G$-spectrum. When $G$ is the trivial group we recover the definition of an orthogonal parameterized spectrum over $B$. 
\end{rmk}

\begin{ex}
If $B$ is a $G$ space and $X$ a retractive $G$-space over $B$, the equivariant suspension spectrum of $X$ over $B$, denoted $\Sigma_B^\infty X$, is the spectrum which at level $n$ is $\Sigma_B^nX$ with structure maps given by the homeomorphisms $\Sigma_B\Sigma_B^nX\to \Sigma_B^n\Sigma_B X \cong \Sigma_B^{n+1} X$. If $X$ is any $G$-space equipped with a map to $B$, we may form a retractive $G$-space by taking the disjoint union with the identity map, $X\amalg B\to B$. We denote this space $X_{+B}$. We denote its parameterized suspension and its parameterized suspension spectrum as $\Sigma_{+B} X$ and $\Sigma^\infty_{+B} X$ respectively. 
\end{ex}

\begin{df}
Let $B$ be a $G$-space and $X$ a parameterized orthogonal $G$-spectrum over $B$. The \textbf{spectrum of sections of $X$}, denoted $\Gamma_B X$, is an orthogonal $G$-spectrum (over $*$) with level $n$ given by the space of $G$-equivariant maps $B\to X_n$ that compose with $X_n\to B$ to the identity. For details regarding the structure maps see \cite[Section 4.3]{malkiewich_convenient}.
\end{df}

\begin{df}
Let $X$ be an orthogonal $G$-spectrum, $V$ a finite dimensional real orthogonal $G$-representation, and $\rho$ the regular representation of $G$. Define the \textbf{$V$th equivariant stable homotopy group of $X$} as 
\[
\pi_V^G(X) = \underset{n\to\infty}{\colim}[S^{n\rho\oplus V},X(n\rho)]_*^G
\]
where the colimit on the right is of based $G$-homotopy classes of continuous $G$-maps.
\end{df}

Now that we have set up the necessary definitions for equivariant parameterized stable homotopy theory, we will move on to discuss fixed point theory. 

\begin{df}\label{df:Reidemeister_trace}
Let $M$ be a closed smooth (or PL) manifold and $f:M\to M$ a continuous map. Define the twisted free loop space of $f$ to be $\Lambda^f M = \{\gamma\colon I\to M\mid f(\gamma(1))=\gamma(0)\}$. Embed $M$ into an open subspace $U$ of $\R^N$ such that there is a retract $p:U\to M$ (making $M$ into a compact ENR) and an $\epsilon$-tube about $M$ contained in $U$. The \textbf{Reidemeister trace of $f$}, denoted $R(f)$, is an element of $\pi_0(\Sigma_+^\infty \Lambda^f M)$ which we define to be the following colimit of homotopy classes of based maps:
\[
\pi_0(\Sigma_+^\infty \Lambda^f M)=\underset{N\to \infty}{\colim} [S^N, S^N \sma_+(\Lambda^f M)]_*.
\]
Specifically, we define $R(f)$ as
\begin{align*}
S^N & \to S^N_\epsilon \wedge_+ (\Lambda^f M)\\
v & \mapsto \begin{cases}
(v-f(p(v))) \wedge \gamma_{f(p(v)),v} & \text{ if $v\in V$ and $\|v-f(p(v))\| \leq \epsilon$}\\
* & \text { else}.
\end{cases}
\end{align*}
Here the source $S^N$ is the one point compactification of $\R^N$ and the target $S^N_\epsilon$ is a sphere of radius $\epsilon$ obtained by quotienting the complement of an open ball of radius $\epsilon$ in $\R^N$:
\[
S^N_\epsilon = \R^N / (\R^N - B_\epsilon).
\]
The path $\gamma_{f(p(v)),v}$ is defined by
\[
\gamma_{f(p(v)),v}(t) = p[(1-t)f(p(v)) + tv]
\]
\end{df}

The map $p$ is defined on the straight line from $f(p(v))$ to $v$ because of our condition on $\epsilon$. As mentioned in the introduction, $R(f)$ is a complete obstruction to the removal of fixed points from $f$ as long as $\dim(M)\geq 3$. In other words $f$ is homotopic to a map with no fixed points if and only if $R(f)$ is null-homotopic for $\dim(M)\geq 3$.

\begin{df}
Given $f\colon M\to M$ we define the \textbf{Lefschetz trace of $f$}, denoted $L(f)$, to be the composition of $R(f)$ with the map induced by $\Lambda^f M \to *$. This is an element of $\pi_0(\Sph) = \Z$.
\end{df} 

We now discuss periodic point invariants. Note that $M^{\times n}$ is a $C_n$-manifold where the action cycles the factors of the product. The $n$th Fuller trace of $f\colon M\to M$, denoted $R(\Psi^nf)^{C_n}$, is given by applying the $C_n$-equivariant Reidemeister trace to the map $\Psi^n(f)$ described in the introduction. The equivariant Reidemeister trace is defined the same as the non-equivariant version except that all the spaces are equipped with a group action, and all the maps are equivariant. In particular, the  Fuller trace, $R(\Psi^n f)^{C_n}$,  is an element of 
\[
\pi_0^{C_n}=\underset{n\to\infty}{\colim} [S^{n\rho},S^{n\rho}\wedge_+ (\Lambda^{\Psi^nf} M^{\times n})]_*
\]
where $\rho$ is the regular $C_n$-representation and the colimit is of $C_n$-equivariant based maps. We will introduce a specific formula for this later in the paper. For a general formula, see \cite[Section 2]{mp2}.

The fiberwise Reidemeister trace of a family of endomorphisms $f\colon E\to E$ over $B$ is a map of parameterized spectra
\[
R_B(f)\colon B\times \Sph \rightarrow \Sigma_{+B}^\infty \Lambda_B^f E.
\]
where 
\[
\Lambda_B^f E = B\times_{\Lambda^f B}\Lambda^f E
\]
is the space of fiberwise twisted free loops. More concretely, the fiberwise Reidemeister trace can be defined as a map
\[
B\times S^N \to \Sigma_{+B}^N \Lambda_B^f E. 
\]
The construction of the fiberwise Reidemeister trace is given by first fiberwise embedding $E$ into $B\times \R^N$ and then taking the Reidemeister trace in each fiber. Note that $R_B(f)$ evaluated at $\{b\}\times \Sph$ is precisely $R(f_b)$. The fiberwise Lefschetz trace, $L_B(f)$ is constructed by composing $R_B(f)$ with the map induced by $\Lambda_B^f E \to B$.

\begin{df}
The \textbf{fiberwise Fuller trace of $f$}, denoted $R_{B}(\Psi_B^n f)^{C_n}$, is defined by taking the $C_n$-equivariant fiberwise Reidemeister trace of the fiberwise Fuller construction of $f$. The \textbf{fiberwise Fuller construction of $f$} is defined in the same way as the usual Fuller construction except that it is a self map of the $n$-fold fiber product of $E$ over $B$, denoted by $E^{\times_B^n}$.
\end{df}

One advantage of describing fixed/periodic point invariants as maps of spectra is that we may use Pontryagin-Thom type results to reinterpret these invariants as cobordism classes of framed manifolds. This will be one of the main techniques used in this paper. We will now give the necessary background on cobordism theory before stating the Pontryagin-Thom type result that will be our main computational tool in this paper. 

\begin{df}
Let $X$ be a $G$-space. Let $B$ be a closed $G$-manifold and $\psi\colon X\to B$ a $G$-equivariant map.  A \textbf{singular $G$-manifold over $\psi$} is a triple, $(M,f,\phi)$ where $M$ is a closed $G$-manifold, $f\colon M\to X$ and $\phi\colon M\to B$ are $G$-equivariant maps such that $\phi=\psi\circ f$.
\end{df}

\begin{notn}
Let $M$ be a $G$-manifold and $V$ a finite dimensional orthogonal real $G$-representation. We denote the trivial $V$ bundle on $M$ as $\varepsilon_M(V)= M\times V$. 
\end{notn}

\begin{df}
Let $M$ be a compact $G$-manifold and $\phi\colon M\to B$ a $G$-equivariant map. A \textbf{$V$-framing of $M$ over $B$} is an equivalence class of $G$-vector bundle isomorphisms
\[
TM\oplus \varepsilon_M(\R^{k})\cong \phi^*(TB)\oplus \varepsilon_M(V\oplus\R^k)
\]
where the equivalence relation is given as follows. We say two $G$-vector bundle isomorphisms are equivalent if they are $G$-homotopic over $M$. We also say that a $G$-vector bundle isomorphism
\[
TM\oplus \varepsilon_M(\R^{k+1})\cong \phi^*(TB)\oplus \varepsilon_M(V\oplus\R^{k+1})
\]
is equivalent to 
\[
TM\oplus \varepsilon_M(\R^{k})\cong \phi^*(TB)\oplus \varepsilon_M(V\oplus\R^k)
\]
if it is obtained by extending to the identity in the $(k+1)$ coordinate of every fiber. 
\end{df}

\begin{df}
A \textbf{$V$-framed cobordism element over $\psi:X\to B$} is a singular $G$-manifold over $\psi$ equipped with a $V$-framing over $B$. 
\end{df}

\begin{df}
Two $V$-framed cobordism elements $(M,f,\phi)$ and $(M',f'\phi')$ over $\psi\colon X\to B$ are \textbf{$V$-framed cobordant over $\psi\colon X\to B$} if there exists, $(W,q,p)$, a singular $G$-manifold over $\psi\colon X\to B$ such that:
\begin{description}
\item[(i)] $W$ is a $G$-manifold with a $(V\oplus \R)$-framing over $B$.
\item[(ii)] The boundary of $W$ is $M\amalg M'$. 
\item[(iii)] The induced $V$-framing over $B$ on $\partial W$ restricted to $M$ and $M'$ agree with the framings over $B$ on $M$ and $-M'$ respectively. Here $-M'$ is the framed manifold given by negating the $n^{th}$ coordinate of $\R^n$ in the framing of $M'$. 
\item[(iv)] $q\colon W\to X$ is an equivariant map such that $q|_{M}=f$ and $q|_{M'}=f'$. 
\item[(v)] $p\colon W\to B$ is an equivariant map such that $p|_M=\phi$ and $p|_{M'}=\phi'$.  
\end{description}
We will frequently shorten ``$V$-framed cobordism over $\psi\colon X\to B$" to ``cobordism" when the context is clear.
\end{df}

\begin{notn}
Define $\omega_{V}^G(X\xrightarrow{\psi} B)$ to be the abelian group of $V$-framed $G$-manifolds over $\psi\colon X\to B$ up to cobordism. Note that while the elements of $\omega_{V}^G(X\to B)$ are manifolds of dimension $\dim(V)+\dim(B)$, it is more accurate to think of them as families of manifolds of dimension $\dim(V)$ parameterized over $B$.
\end{notn}

We are now ready to state our main computational tool, the Pontryagin-Thom isomorphism for families of framed $G$-manifolds.

\begin{thm}\cite[Theorem 1.2]{williams_framed_PT}\label{thm:equivariant parameterized PT}
Let $G$ be a finite group, $V$ a finite dimensional orthogonal $G$-representation, $B$ a closed compact $G$-manifold, and $X \to B$ an equivariant Serre fibration. There is an isomorphism of abelian groups
\[
\omega_{V}^G(X\xrightarrow{\psi} B) \cong \pi_V^G\left( \Gamma_B(\Sigma_B^\infty (X\amalg B)) \right)
\]
from the group of $V$-framed $G$-manifolds over $\psi\colon X\to B$ to the $V$th equivariant stable homotopy group of the spectrum of sections of the parameterized suspension spectrum of $X\amalg B$ over $B$. 
\end{thm}

\begin{rmk}
In reference to the above theorem, an equivariant Serre fibration is a map which is a Serre fibration on all of the fixed point subsets. In the above theorem we would usually right derive the spectrum of sections functor $\Gamma_B$. However, the assumption that $X$ is an equivariant Serre fibration over $B$ makes this unnecessary as explained in \cite[Lemma 3.47]{williams_framed_PT}. Throughout this paper, each application of the above theorem will fortuitously be in a setting where the map $X\to B$ happens to be a Serre fibration. 
\end{rmk}

Our strategy in proving \cref{thm:main} will be to interpret various Reidemeister traces as elements of (equivariant) stable homotopy groups of spectra of sections. Then we will reinterpret these traces as framed cobordism classes using \cref{thm:equivariant parameterized PT}. We will then use a mixture of cobordism theory and stable homotopy theory to compute these traces.

\section{The main construction}

In this section we provide an explicit construction that realizes \cref{thm:main}.

\begin{const}\label{const:example}
Let $n\geq 3$, let $B=S^1$ and let $E_\theta$ be the following space for each $\theta\in S^1$: build a necklace of 4 spheres alternating dimension between $S^n$ and $S^{n-1}$ where each sphere is joined to the next by a unit interval as depicted in \cref{fig:necklace of spheres}, such that each interval joins the north pole of a sphere to the south pole of the next sphere in the necklace. In the figure below we have labeled the north poles $n$ and south poles $s$. Define the total space $E$ to be $E_\theta\times B$, the trivial $E_\theta$ bundle over $S^1$.

\begin{figure}[!htb]
\minipage{0.32\textwidth}
  \includegraphics[width=\linewidth]{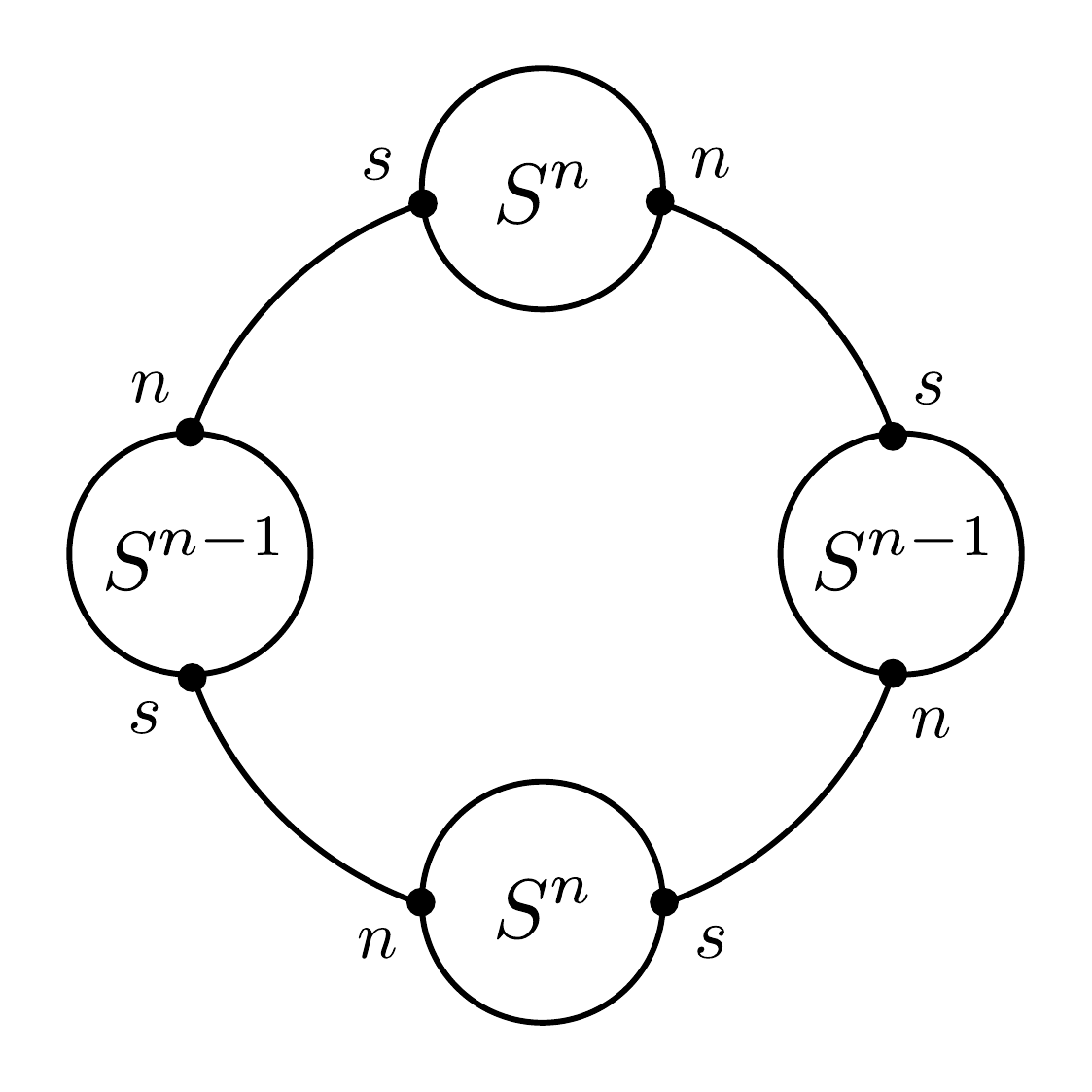}
  \captionsetup{labelsep=newline}
  \caption{The space $E_\theta$}\label{fig:necklace of spheres}
\endminipage\hfill
\minipage{0.32\textwidth}
  \includegraphics[width=\linewidth]{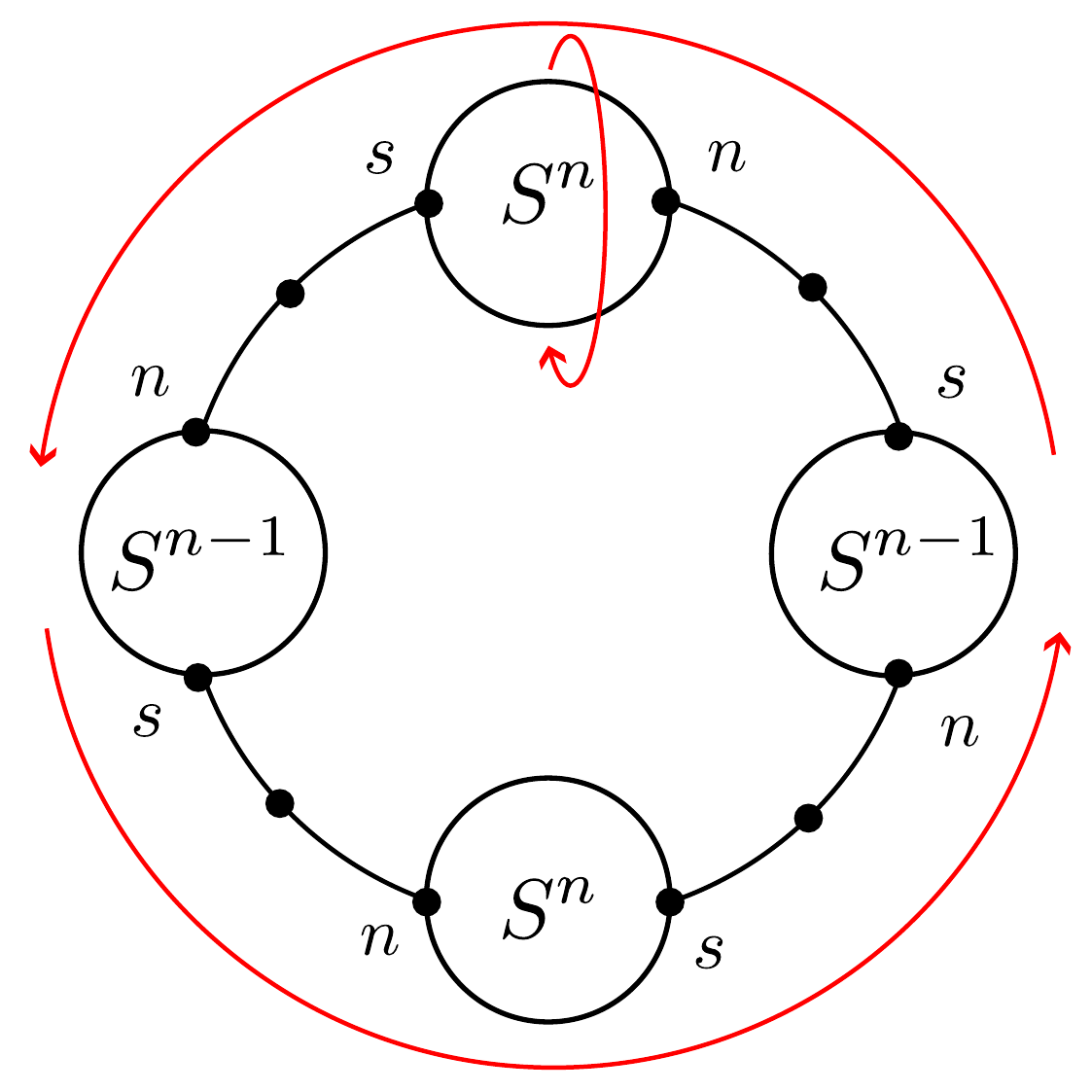}
  \captionsetup{labelsep=newline}
  \caption{The map $f_\theta$}\label{fig:necklace of spheres map}
\endminipage\hfill
\minipage{0.33\textwidth}%
  \includegraphics[width=\linewidth]{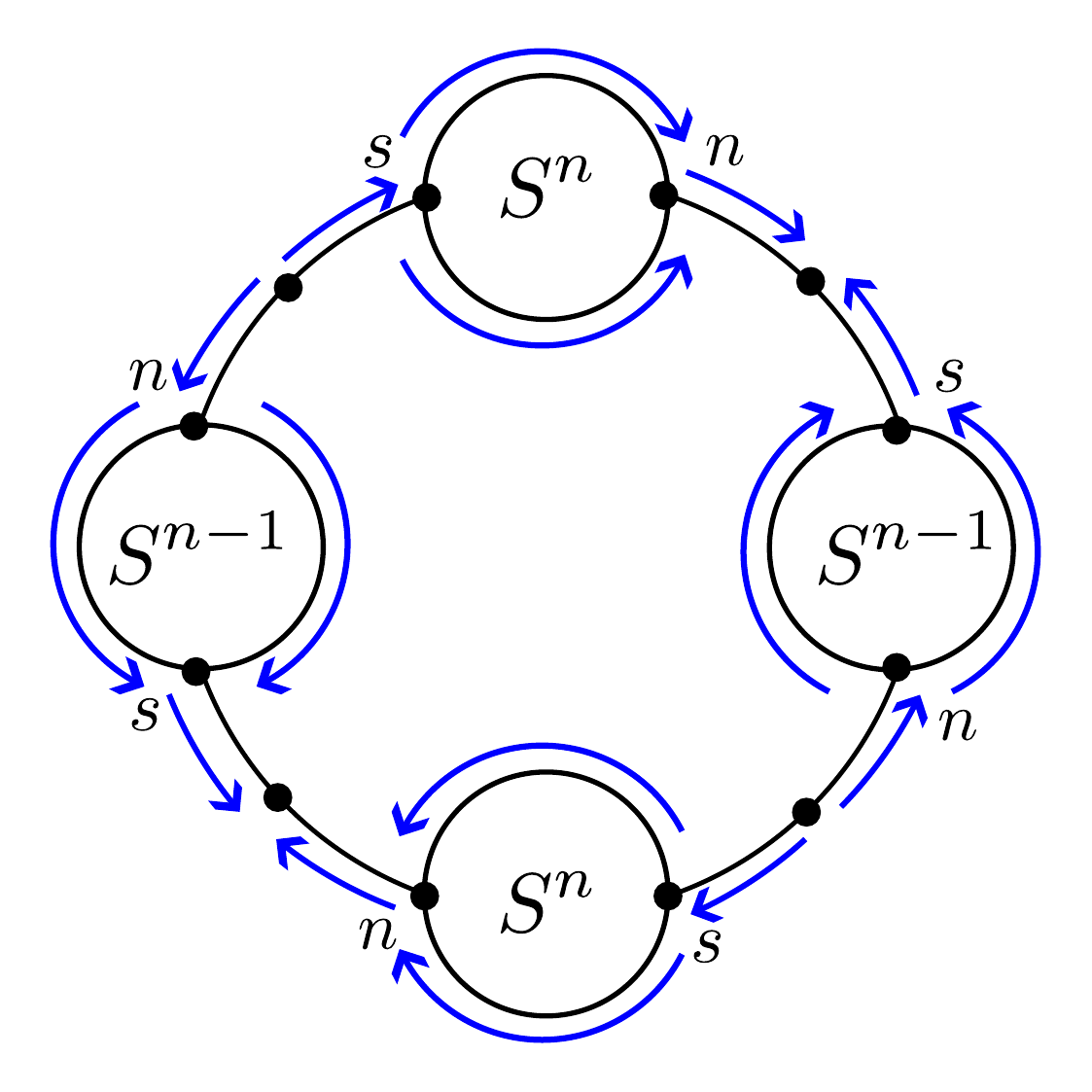}
  \captionsetup{labelsep=newline}  
  \caption{The deformation $h$}\label{fig:necklace of spheres deformation}
\endminipage
\end{figure}

In order to define $f:E\to E$, we need only define $f_\theta$. We define $f_\theta:E_\theta\to E_\theta$ to be the map which rotates the necklace of spheres half way around and spins two coordinates of one copy of $S^n$ by an angle of $\theta$ as depicted in \cref{fig:necklace of spheres map}. In other words, we rotate the necklace half way around then apply the block diagonal matrix, $(a_{ij})$, which is the identity everywhere except for the block $\begin{bmatrix}
a_{2,2}& a_{2,3}\\a_{3,2}&a_{3,3}
\end{bmatrix} = \begin{bmatrix}
\cos(\theta) &-\sin(\theta)\\\sin(\theta) & \cos(\theta)
\end{bmatrix}$ to one copy of $S^n$ (modeled as the unit sphere in $\R^{n+1}$).

While $f$ has no fixed points, $f^2$ has infinitely many non-isolated fixed points. In order to isolate the fixed points of $f_\theta^2$, we will define a map, $h$, homotopic to the identity map. Then $f_\theta^2$ composed with $h$ is homotopic to $f_\theta^2$ with the added benefit that the composition has isolated fixed points. We will define $h$ piecewise as depicted in \cref{fig:attracting interval,fig:repelling interval,fig:n-sphere,fig:(n-1)-sphere}. We have denoted the fixed points by the black dots between the arrows. When it becomes necessary later in the paper, we will label each fixed point. 

\begin{figure}[!htb]
\minipage{0.24\textwidth}
  \includegraphics[width=\linewidth]{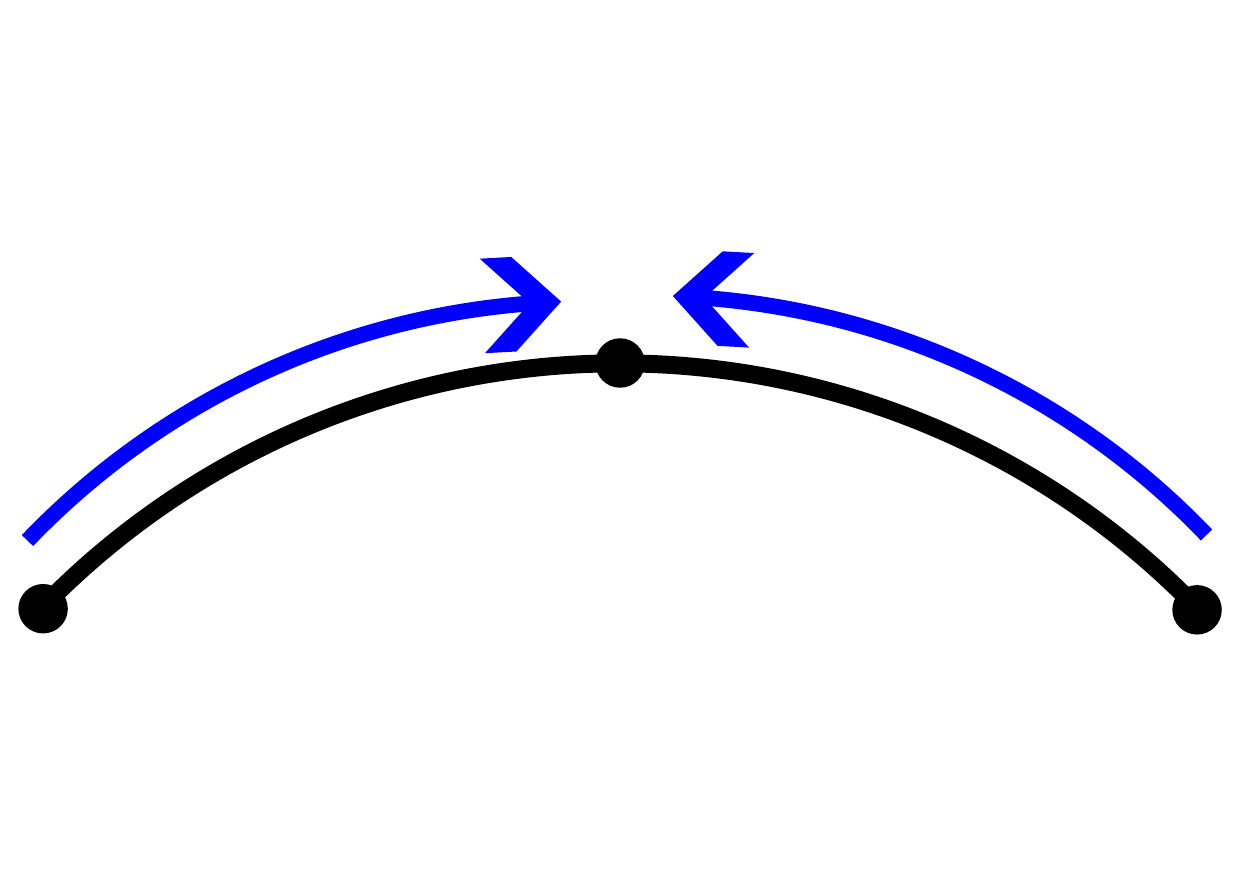}
  \caption{}\label{fig:attracting interval}
\endminipage\hfill
\minipage{0.24\textwidth}
  \includegraphics[width=\linewidth]{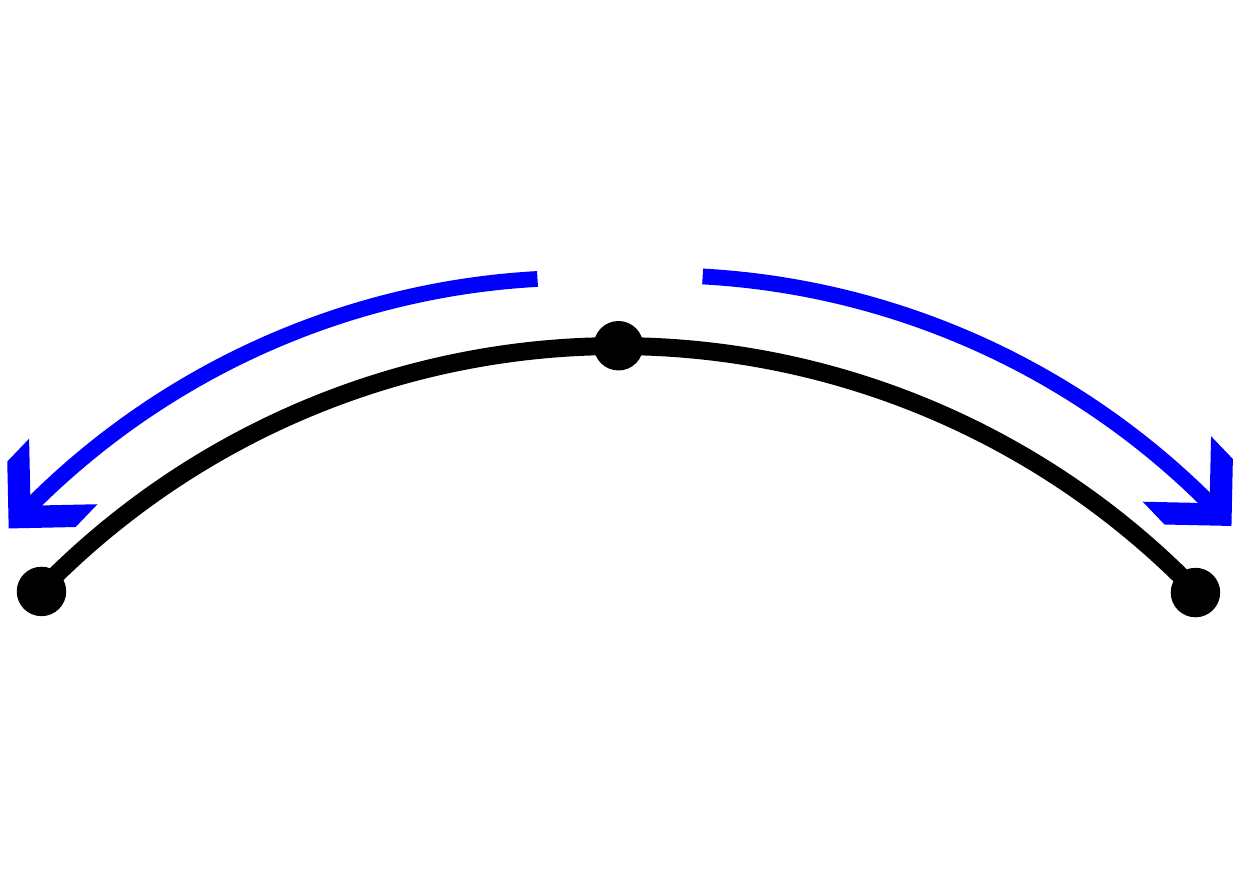}
  \caption{}\label{fig:repelling interval}
\endminipage\hfill
\minipage{0.24\textwidth}%
  \includegraphics[width=\linewidth]{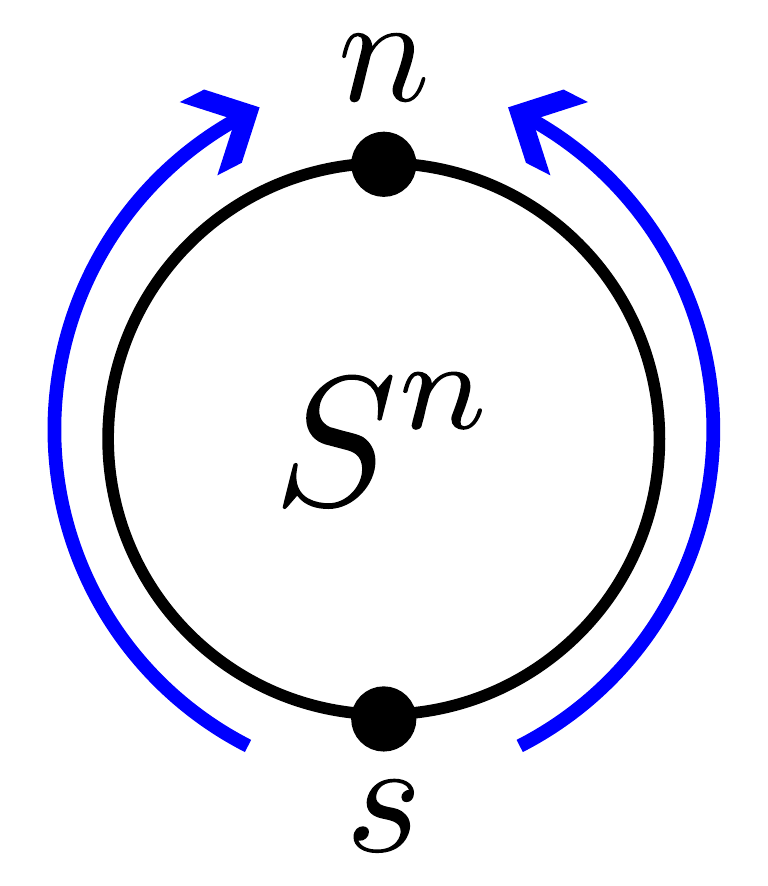}
  \caption{}\label{fig:n-sphere}
\endminipage
\minipage{0.24\textwidth}%
  \includegraphics[width=\linewidth]{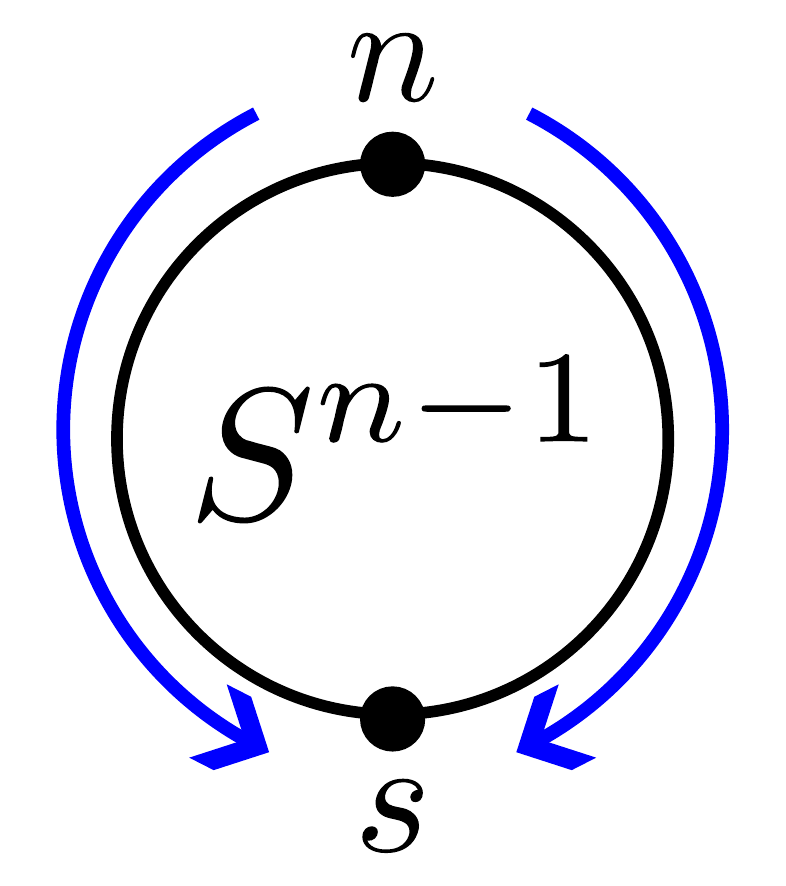}
  \caption{}\label{fig:(n-1)-sphere}
\endminipage
\end{figure}

On each copy of $I=[0,1]$ of the form shown in \cref{fig:attracting interval} we apply the map
\begin{align*}
I & \longrightarrow I\\
t & \mapsto \begin{cases}
t[-(t-1/4)^2+(1+1/16)] & \text{ if } 0\leq t\leq 1/2\\
t[(t-3/4)^2+(1-1/16)] & \text{ if } 1/2\leq t\leq 1
\end{cases}
\end{align*}

On each copy of $I=[0,1]$ of the form shown in \cref{fig:repelling interval} we apply the map
\begin{align*}
I & \longrightarrow I\\
t & \mapsto \begin{cases}
t[(t-1/4)^2+(1-1/16)] & \text{ if } 0\leq t\leq 1/2\\
t[-(t-3/4)^2+(1+1/16)] & \text{ if } 1/2\leq t\leq 1
\end{cases}
\end{align*} 

On each copy of $S^n$ of the form shown in \cref{fig:n-sphere} we define $h'$ as follows. Model $S^n$ as the unit sphere in $\R^{n+1}$. Define a smooth vector field
\begin{align*}
X\colon S^n & \to T\R^{n+1}\\
x& \mapsto (x,(1,0,\dots,0)).
\end{align*}
Let $\widetilde{X}\colon S^n\to TS^n$ be the smooth vector field given by orthogonally projecting each vector in the image of $X$ to the tangent space of $S^n$. We think of $\tilde{X}$ as being the consistently northward pointing vector field on $S^n$. The map on $S^n$ is then defined by flowing along $\widetilde{X}$ to time $t=1$. 

On each copy of $S^{n-1}$ depicted in \cref{fig:(n-1)-sphere} we apply the same construction as in $S^n$ except we choose a southward pointing vector field as opposed to a northward pointing vector field. 

These maps then assemble to the piecewise map $h\colon E_\theta\to E_\theta$ homotopic to the identity.  
\end{const}

\section{Proof of the main theorem}

\subsection{Fiberwise Reidemeister traces of iterates}\label{reidemeister zero}

In this section we will prove the first part of \cref{thm:main}. Namely, that the fiberwise Reidemeister traces of $f$ and $f^2$ from \cref{const:example} are trivial.

\begin{lem}\label{lem:trivial f}
The fiberwise Reidemeister trace
\[
R_{S^1}(f) \colon S^1\times \Sph \to \Sigma_{+S^1}^\infty \Lambda_{S^1}^{f} E
\]
of the map of \cref{const:example} is trivial.
\end{lem}

\begin{proof}
Since $f\colon E\to E$ has no fixed points, the $v-f(p(v))$ term in the formula for $R_{S^1}(f)$ is never zero. Then $R_{S^1}(f)$ can be deformed to zero in each fiber since it is never surjective.
\end{proof}

In order to prove that $R_{S^1}(f^2)$ is trivial, we will need the following lemma. Consider the spectrum $\Gamma_{S^1}\Sigma_{+S^1}^\infty S^1$ where both copies of $S^1$ map to the base space $S^1$ via the identity. 

\begin{lem}\label{lem:split section}
The spectrum of sections, $\Gamma_{S^1}\Sigma_{+S^1}^\infty S^1$, is stably equivalent to the cotensor spectrum $F(S^1_+,\Sph)$ and thus, to $\Sph\vee  \Sph^{-1}$. Moreover, the $\Sph$ summand corresponds to sections of $\Sigma_{+S^1}^\infty S^1$ which are constant in each fiber.
\end{lem}

\begin{proof}
Spectrum level $k$ of $\Gamma_{S^1}\Sigma_{+S^1}^\infty S^1$ is the space of sections $\Gamma_{S^1}S^k\overline{\wedge}_{+S^1}S^1$. Moreover, $S^k\overline{\wedge}_{+S^1}S^1 \simeq S^k\times S^1$. Therefore, a section of this space over $S^1$ is just a map $S^1_+\to S^k$. This identification commutes with the structure maps of $\Gamma_{S^1} \Sigma_{+S^1}^\infty S^1$ so that it induces an equivalence 
\[
\Gamma_{S^1}\Sigma_{+S^1}^\infty S^1 \simeq F(S^1_+,\Sph).
\]

Now observe that we have a retract of spaces 
\[
*_+ \to S^1_+ \to *_+.
\]
This induces a retract of cotensor spectra
\[
F(*_+,\Sph) \to F(S^1_+,\Sph) \xrightarrow{j} F(*_+,\Sph).
\]
Therefore, 
\[
F(S^1_+,\Sph)\simeq F(*_+,\Sph)\vee Fj
\]
where $Fj$ is the fiber of $j$. We may identify $F_j$ with $F(S^1,\Sph)$. Then identifying $F(*_+,\Sph)\simeq \Sph$ and $F(S^1,\Sph)\simeq \Sph^{-1}$, we obtain
\[
\Gamma_{S^1} \Sigma_{+S^1}^\infty S^1\simeq \Sph\vee \Sph^{-1}.
\]
The last part of the lemma statement follows from the fact that the map $j$ is induced by mapping $S^1$ to a single point. 
\end{proof}

We must now fix notation for various framings of $S^1$. By a framing, we will mean a trivialization of $\nu(S^1,\R^3)$ where we embed $S^1$ as the unit circle in the $(x,y,0)$-plane.
 
\begin{df}\label{df:standard framings on S^1}
The \textbf{standard framing on $S^1$} is given by identifying $\nu(S^1,\R^3)$ with $S^1\times \R^2$ as follows: send the outward pointing unit vector in the radial direction to the first standard basis vector in each fiber, and send the unit vector pointing in the positive $z$-direction to the second standard basis vector in each fiber. We denote $S^1$ equipped with this framing as $S^1_{\std}$. The \textbf{negative standard framing on $S^1$} is constructed similarly except we use the inward pointing unit vector in the radial direction instead of the outward pointing one. We denote $S^1$ equipped with this framing as $S^1_{-\std}$. 
\end{df}

\begin{df}\label{df:Hopf framings}
The \textbf{positive and negative Hopf framings of $S^1$} are given as follows: compose the positive and negative standard framings on $S^1$ with the map $S^1\times\R^2\to S^1\times\R^2$ induced by the map $S^1\to O(2)$ sending $\theta$ to $\begin{bmatrix}
\cos(\theta) & -\sin(\theta)\\
\sin(\theta)& \cos(\theta)
\end{bmatrix}$. We denote these framings by $S^1_{\eta}$ and $S^1_{-\eta}$ respectively. 
\end{df}

The names and corresponding notation in the above definition is chosen to reflect the fact that under the identification $\omega_1(*\to *)\cong\pi_1(\Sph)$ the manifold $S^1$ with the Hopf framing is mapped to a suspension of the Hopf fibration. 

We are now ready to prove the following proposition. 

\begin{prop}\label{prop:fiberwise Lefschetz cobordism class}
Under the isomorphism of \cref{thm:equivariant parameterized PT}, the fiberwise Lefschetz trace
\[
L_{S^1}(f^2) :S^1\times \Sph \to \Sigma_{+S^1}^\infty \Lambda_{S^1}^{f^2} E \to \Sigma_{+S^1}^\infty S^1
\]
corresponds to 
\[
2\cdot S^1_{\std} \amalg 2\cdot S^1_{-\std} \amalg 2\cdot S^1_{(-1)^{n-1}\std} \amalg 2\cdot S^1_{(-1)^n\eta}
\]
in $\omega_0(S^1\xrightarrow{\id} S^1)$.
\end{prop}

\begin{proof}
Let $F = h\circ f^2$ where $h$ is as in \cref{const:example}. Since $h$ is fiberwise homotopic to the identity, $L_{S^1}(f^2) = L_{S^1}(F)$. Thus, it suffices to compute the framed manifold associated to $L_{S^1}(F)$. As in \cite[Proposition 3.40]{williams_framed_PT}, the framed manifold associated to $L_{S^1}(F)$ is given by projecting away the $S^1$ coordinate and then taking the preimage of 0 (after making the map smooth and transverse to 0). The framing on this manifold, henceforth referred to as $M$, is given by pulling back the standard basis vectors at the tangent space at 0 through the map $L_{S^1}(F)$. 

Unwinding the definition of $L_{S^1}(F)$, we see that the underlying manifold of $M$ is precisely the fixed points of $F$. These fixed points are depicted in \cref{fig:necklace of spheres deformation}. Moreover, the framing of $M$ is given by the local index of $F$ at each of its fixed points. Thus, the underlying manifold of $M$ is the disjoint union of 12 copies of $S^1$ which traverse the base space exactly once. Note that we will remove some of these copies of $S^1$ by further perturbing the map $h$.

For both copies of $S^1$ appearing fiberwise as the center point of \cref{fig:attracting interval}, we have a copy of $S^1$ with the standard framing.

For both copies of $S^1$ appearing fiberwise as the center point of \cref{fig:repelling interval}, we have a copy of $S^1$ with the negative standard framing.

Now examine the copies of $S^1$ given fiberwise by the north and south poles of \cref{fig:(n-1)-sphere}. Recall that here the map $F$ is given by flowing toward the south pole. Since we need only care about the behavior of $F$ near a fixed point, we can modify $F$ in the neighborhood of a fixed point without changing the Lefschetz trace, and thus also the framed cobordism class of $M$. Choose a small neighborhood of the south pole of $S^{n-1}$, contract it to the south pole, and then translate the south pole along the adjoining interval slightly. This modification removes the fixed points at the south poles of both copies of $S^{n-1}$ so that the corresponding copies of $S^1$ in $M$ vanish.  

Now working locally around the north pole of $S^{n-1}$, we can first contract the adjoining interval to the north pole so that we have a fixed point around which $F$ repels in all $(n-1)$ directions of $S^{n-1}$. Since $F$ is repelling in every direction near the fixed point we may locally model $F$ as a map on $\R^{n-1}$ given by $x\mapsto \lambda x$ with $\lambda>1$. Therefore, the map $x\mapsto x-F(x)$ is homotopic to the antipodal map, the degree of which depends only on the parity of $n-1$. Thus, the fixed point at the north pole in each copy of $S^{n-1}$ gives rise to a copy of $S^1_{(-1)^{(n-1)}\std}$.

Finally, examine the copies of $S^1$ given fiberwise by the north and south poles of \cref{fig:n-sphere}. Begin by examining the north pole (which the map flows toward). Similarly to the previous case, we contract a small neighborhood of the north pole to the north pole and then translate slightly along the adjoining interval. This removes the fixed point at the north pole so that we may disregard both of the corresponding copies of $S^1$ in $\omega_0(S^1\to S^1)$. 

At the south pole of $S^n$ we again work locally. We begin by contracting the adjoining interval to the south pole. Recall that near the south pole in the fiber above $\theta$ the map $F$ spins two of the coordinates of $S^n$ by the angle $\theta$ and then flows away from the south pole in every direction. We can then see that each of the south poles of our two copies of $S^n$ give rise to $S^1_{(-1)^n\eta}$. Note that the framing is actually the Hopf framing stabilized by the identity, which we consider equivalent to the Hopf framing.

Note that in this computation the reference map from each framed manifold to the base space is the identity.
\end{proof}

\begin{prop}\label{prop:trivial lefschetz trace}
The fiberwise Lefschetz trace
\[
L_{S^1}(f^2) :S^1\times \Sph \to \Sigma_{+S^1}^\infty \Lambda_{S^1}^{f^2} E \to \Sigma_{+S^1}^\infty S^1
\]
of the square of the map of \cref{const:example} is trivial.
\end{prop}
\begin{proof}
By \cref{thm:equivariant parameterized PT} and \cref{prop:fiberwise Lefschetz cobordism class}, it suffices to show that 
\[
2\cdot S^1_{\std} \amalg 2\cdot S^1_{-\std} \amalg 2\cdot S^1_{(-1)^{n-1}\std} \amalg 2\cdot S^1_{(-1)^n\eta}
\]
is null-cobordant in $\omega_0(S^1\to S^1)$.

Each copy of $S^1$ with the standard framing is inverse to a copy of $S^1$ with the negative standard framing. The cobordism is given by taking $S^1\times I$ and bending it into a horseshoe shape. The reference map $S^1\times I\to S^1$ is given by projecting to $S^1$. 

We now claim that $2\cdot S^1_{(-1)^n\eta}$ is cobordant over $S^1$ to $2\cdot S^1_{(-1)^n\std}$. The pants cobordism gives an equivalence over $S^1$ from $2\cdot S^1_{(-1)^n\eta}$ to $S^1_{(-1)^n\std}$ which is then cobordant to $\cdot S^1_{(-1)^n\std}$ via the upside down pants. 

Therefore, 
\[
2\cdot S^1_{(-1)^{n-1}\std} \amalg 2\cdot S^1_{(-1)^n\eta} \simeq  2\cdot S^1_{(-1)^{n-1}\std} \amalg 2\cdot S^1_{(-1)^{n}\std}\simeq \emptyset.
\]

Therefore, $L_{S^1}(F)$ is null homotopic as desired. 
\end{proof}

\begin{lem}
The composition
\[
S^1\times \Sph \xrightarrow{R_{S^1}(f^2)} \Sigma_{+S^1}^\infty \Lambda_{S^1}^{f^2}E \to \Sigma_{+S^1}^\infty L_{S^1} T^2
\]
is trivial. Here the second map in the composition is induced by collapsing $E_\theta\to S^1$ by collapsing each of the high-dimensional spheres to a point. 
\end{lem}

\begin{proof}
Along the above composition, $f^2$ becomes the identity. Therefore, $\Lambda_{S^1}^{f^2}E$ becomes $L_{S^1}T^2$, the fiberwise free loop space on $T^2$ over $S^1$. 

Recall that in \cref{prop:trivial lefschetz trace} we showed the composition
\[
S^1\times \Sph \to \Sigma_{+S^1}^\infty \Lambda_{S^1}^{f^2}E \to \Sigma_{+S^1}^\infty L_{S^1} T^2 \to \Sigma^\infty_{+S^1} S^1
\]
is trivial by computing the cobordism class of the corresponding framed manifold over $S^1$. This framed manifold was a disjoint union of copies of $S^1$ each of which admitted a reference map to $\Lambda_{S^1}^{f^2} E$ as the inclusion of the fixed points of $F$. When we collapsed $E$ to $S^1$ by collapsing $E_\theta$ to $*$, the reference maps all became the identity map on $S^1$. However, we now collapse $E$ to $T^2$ and along this map the reference maps include each copy of $S^1$ into $T^2$ so that each copy of $S^1$ traverses the base once and the fiber zero times. The components of the reference map $\amalg_8 S^1 \to L_{S^1}T^2$ are then given by identifying the images of $S^1$ in $T^2$ with fiberwise constant free loops. Then each of the components of $\amalg_8 S^1\to L_{S^1} T^2$ are fiberwise homotopic. By fixing a homotopy between our reference maps, the calculation of \cref{prop:trivial lefschetz trace} shows that $(\amalg_8 S^1 \to L_{S^1}T^2)=0$ in $\omega_0(L_{S^1}T^2\to S^1)$. Thus, applying \cref{thm:equivariant parameterized PT} proves the desired claim. 
\end{proof}

\begin{prop}\label{prop:trivial f^2}
The fiberwise Reidemeister trace
\[
R_{S^1}(f^2) :S^1\times \Sph \to \Sigma_{+S^1}^\infty \Lambda_{S^1}^{f^2} E
\]
of the square of the map of \cref{const:example} is trivial.
\end{prop}

\begin{proof} 
First fix basepoints $e_0\in E_\theta$ and $x_0\in S^1$ so that the map $h_\theta:E_\theta\to S^1$ which collapses each high dimensional sphere to a point is based. Now observe that this map is $(n-1)$-connected since we assumed that $n\geq 3$. Thus, the fiberwise map $c:E\to T^2$ is also $(n-1)$-connected since $E$ and $T^2$ are both trivial $S^1$ bundles. Along $c_\theta$ the map $f^2$ is the identity. We therefore have a map of fiber sequences,
\begin{figure}[H]
\center
\begin{tikzcd}
\Omega_{S^1,(e_0,f^2(e_0))}E \arrow[r]\arrow[d,"g"]& \Lambda_{S^1}^{f^2} E \arrow[d]\arrow[r,"ev_0"]& E\arrow[d,"c"]\\
\Omega_{S^1,x_0}T^2\arrow[r] & L_{S^1} T^2 \arrow[r,"ev_0"] & T^2
\end{tikzcd}
\end{figure}  
where $\Omega_{S^1,(e_0,f^2(e_0))}E$ is the space of fiberwise paths from $e_0$ to $f^2(e_0)$ in $E$. We now claim that the map $g$ in the above diagram is $(n-2)$-connected. Since $E_\theta$ and $S^1$ are both path connected, we can find a path $\gamma_{e_0,f^2(e_0)}$ from $e_0$ to $f^2(e_0)$ in $E_\theta$. We then have a map from the fiberwise based loop space $\Omega_{S^1,e_0} E$ to $\Omega_{S^1,(e_0,f^2(e_0))}E$ given by concatenating a fiberwise based loop with $\gamma_{e_0,f^2(e_0)}$. This map is a weak equivalence by contracting a neighborhood of $\gamma_{e_0,f^2(e_0)}$ to $e_0$. We then have the following commutative diagram:
\begin{figure}[H]
\center
\begin{tikzcd}
\Omega_{S^1,e_0}E \arrow[r,"\simeq"]\arrow[dr,"\Omega_{S^1}c"']& \Omega_{S^1,(e_0,f^2(e_0))}E\arrow[d,"g"]\\
{}&\Omega_{S^1,x_0} T^2.
\end{tikzcd}
\end{figure}
Since $c$ is $(n-1)$-connected, $g$ must be $(n-2)$-connected. Therefore, the above diagram of fiber sequences shows that $\Lambda_{S^1}^{f^2} E\to L_{S^1}T^2$ is $(n-2)$-connected. 

We may therefore identify 
\[
\omega_{0}(\Lambda_{S^1}^{f^2}E\to S^1) \cong \omega_{0}(L_{S^1} T^2\to S^1)
\] 
as $0<n-2$. Since the framed manifold associated to $R_{S^1}(f^2)$ in $\omega_{0}(L_{S^1}T^2\to S^1)$ is null-cobordant, we may therefore conclude that the framed manifold associated to $R_{S^1}(f^2)\in \omega_0(\Lambda_{S^1}^{f^2}E)$ is also null-cobordant. Then again applying \cref{thm:equivariant parameterized PT}, the claim follows. 
\end{proof}

\subsection{The fiberwise Fuller trace}\label{fuller nonzero}

In this section we will show that the fiberwise Fuller trace, $R_{S^1}(\Psi^2 f)^{C_2}$, of \cref{const:example} is nontrivial. We begin with a cobordism computation. 

\begin{prop}\label{prop:cobordism class of fiberwise Fuller trace}
Under the isomorphism of \cref{thm:equivariant parameterized PT} the composition
\[
S^1\times \Sph \xrightarrow{R_{S^1,C_2}(\Psi_{S^1}^{2}f)} \Sigma_{+S^1}^\infty \Lambda_{S^1}^{\Psi_{S^1}^{2}f}(E\times_{S^1}E) \to \Sigma_{+S^1}^\infty S^1
\]
of the fiberwise Fuller trace of $f$ with the map induced by $E\times_{S^1} E\to S^1$ corresponds to
\[
C_2\times S^1_{-\std}\amalg C_2\times S^1_{\std}\amalg C_2\times S^1_{(-1)^{n}\eta}\amalg C_2\times S^1_{(-1)^{n-1}\std}
\]
in $\omega_{0}^{C_2}(S^1\to S^1)$. Here the subscripts on $S^1$ denote the fact that $S^1$ is equipped with framings described in \cref{df:standard framings on S^1} and \cref{df:Hopf framings}.
\end{prop}

\begin{proof}
Again let $F=h\circ f$ where $h$ and $f$ are as in \cref{const:example}. Since $F$ and $f$ are fiberwise homotopic, it suffices to compute the framed $C_2$-cobordism class associated to the above composition but with $f$ replaced with $F$.
 
In \cite{mp2} Malkiewich and Ponto show that the fiberwise Fuller trace of $F$ is a map 
\[
S^1\times \Sph \to \Sigma_{+S^1}^\infty \Lambda_{S^1}^{\Psi_{S^1}^{2}F}E\times_{S^1}E
\]
given by 
\begin{align*}
S^1\times S^{2N} & \to S^N\barsmash S^N\barsmash \left( \Lambda_{S^1}^{\Psi_{S^1}^{2}F}E\times_{S^1}E \coprod S^1  \right)\\
(\theta,v_1,v_2) & \mapsto (v_1-F_\theta (p_\theta (v_2))\wedge v_2-F_\theta (p_\theta (v_1))\wedge(\gamma_{F_\theta( p_\theta (v_2)), v_1},\gamma_{F_\theta (p_\theta (v_1)),v_2}))
\end{align*}
where $p_\theta$ is the projection of the normal bundle to the fiber at $\theta$ and $\gamma_{F_\theta (p_\theta( v_i)), v_j}$ is a path in $E$ from $F_\theta( p_\theta( v_i))$ to $v_j$.

Using \cref{thm:equivariant parameterized PT}, identify the above composition with an element of $\omega_0^{C_2}(S^1\to S^1)$ which we denote as $M$. The proof of \cref{thm:equivariant parameterized PT} (which appears as \cite[Theorem 3.60]{williams_framed_PT}) shows that the underlying manifold of $M$ is given by the disjoint union of the fixed points and 2-periodic points of $F$. Since $F$ has no fixed points, $M$ is precisely the 2-periodic points of $F$ where the $C_2$ action swaps the two coordinates in $E\times_{S^1}E$.

It remains to compute the $C_2$-framing on $M$. We will use the following labeling scheme for the fixed points of $F$ in both factors of $E_\theta$ in $E\times_{S^1} E$.

\begin{figure}[H]
\center
\includegraphics[scale=.4]{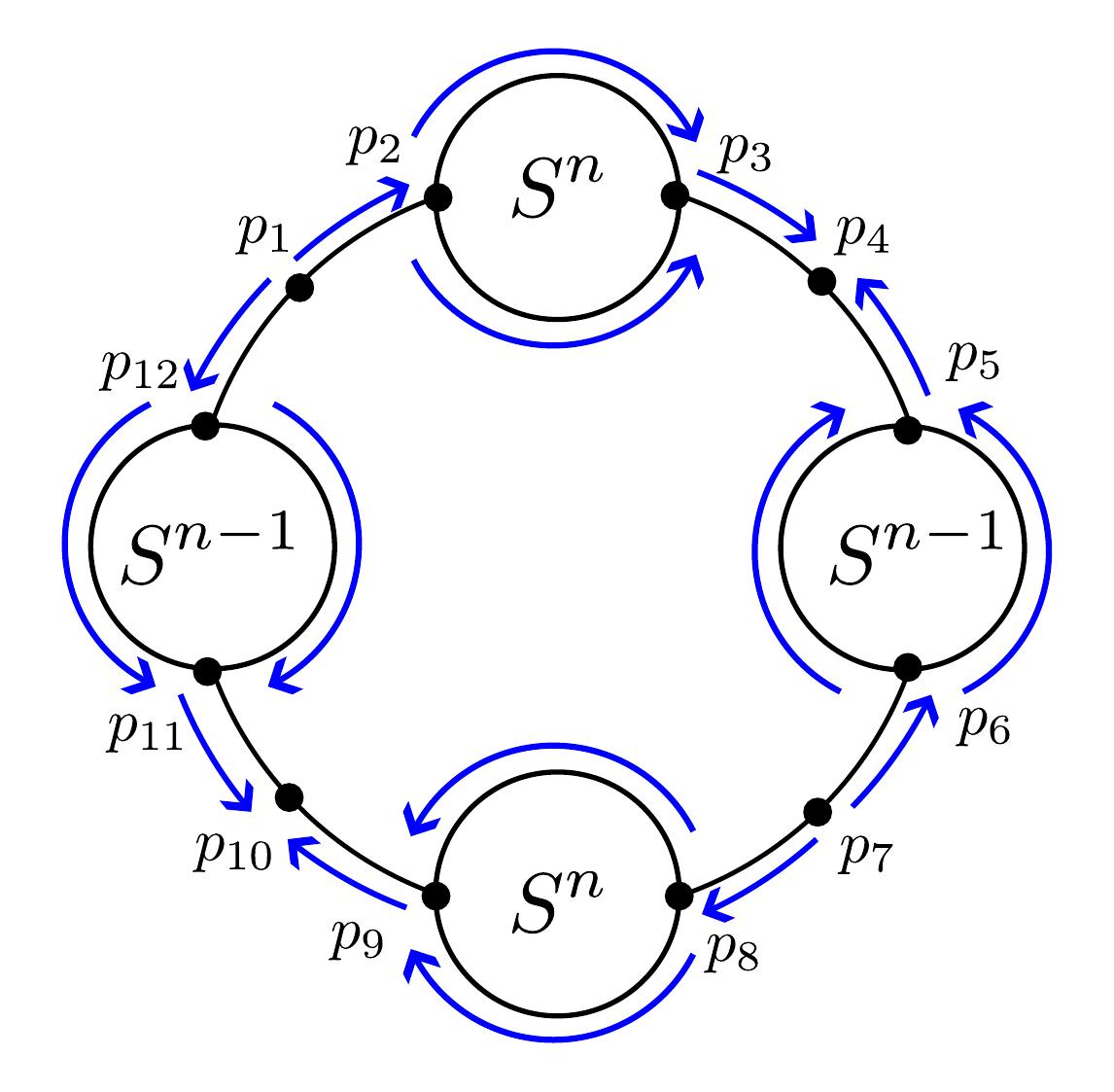}
\caption{Labeling scheme for fixed points}
\end{figure}

As a $C_2$-manifold, $M$ is the disjoint union of 6 copies of $C_2\times S^1$. In terms of the above labeling scheme, $M$ is
\begin{align*}
(p_1,p_7)\times S^1 & \longleftrightarrow (p_7,p_1)\times S^1 & (p_4,p_{10})\times S^1 & \longleftrightarrow (p_{10},p_4)\times S^1\\
(p_2,p_8)\times S^1 & \longleftrightarrow (p_8,p_2)\times S^1 & (p_5,p_{11})\times S^1 & \longleftrightarrow (p_{11},p_5) \times S^1\\
(p_3,p_9)\times S^1 & \longleftrightarrow (p_9,p_3)\times S^1 & (p_6,p_{12})\times S^1 & \longleftrightarrow (p_{12},p_6) \times S^1
\end{align*}
where the double sided arrow denotes the $C_2$-action. 

Note that $C_2$-equivariance ensures that the framing on $C_2\times S^1$ is determined entirely by the framing on one of the copies of $S^1$. For the rest of the proof we will analyze the local behavior of $F$ around its 2-periodic points. It will be more technically convenient to analyze the local behavior of $\Psi^{2}_{S^1} F$ around its fixed points. Since these are equivalent, it presents no problem to make this shift in perspective.  

In a neighborhood of $(p_1,p_7)\times S^1$, the map $\Psi_{S^1}^{2}F$ is constant across fibers so $(p_1,p_7)\times S^1$ is equipped with either the positive or negative standard framing. Here $\Psi_{S^1}^2F$ repels in 2 directions as in the proof of \cref{prop:trivial lefschetz trace} then interchanges the two directions. The sign of the framing is given by the degree of $(\id-\Psi_{S^1}^2 f)$. In local coordinates, this is given by the determinant of $\begin{bmatrix}
1&-\lambda\\-\lambda&1\end{bmatrix}$ for some $\lambda>1$. Since the determinant is negative, we conclude that, as a framed manifold, $(p_1,p_7)\times S^1$ is $S^1_{-\std}$. 

In a neighborhood of $(p_2,p_8)\times S^1$, the map $\Psi_{S^1}^{2} F$ twists two of the $n$ directions of the copy of $S^n$ associated to $p_2$ by $\theta$ in the fiber $E_\theta$. The map also repels in the $n$ directions of $S^n$ associated to $p_2$ and the $n$ directions of $S^n$ associated to $p_8$ then swaps the coordinates associated to $p_2$ with the coordinates associated to $p_8$. The only local behavior here that is non-constant across the fibers is the twisting action which is explicitly described in \cref{const:example}. Thus, $(p_2,p_8)\times S^1$ is one of $S^1_{\pm \eta}$ as a framed manifold. Note that it is actually a stabilization of $\pm\eta$ which we consider to be equivalent to $\pm\eta$. The sign will again be given by the sign of the determinant of $(\id-\Psi_{S^1}^2F)$ after writing this map in local coordinates. In local coordinates, $\Psi_{S^1}^2F$ may be expressed as the $2n\times 2n$ block matrix 
\[
M=\left[\begin{array}{c|c}
0_n&\lambda I_n\\ \hline
\lambda(a_{i,j}) & 0_n
\end{array}\right]
\]
where $\lambda>1$ and $(a_{i,j})$ is the matrix of \cref{const:example}. Then,
\[
\det(I-M)=(1-\lambda^2)^{n-2} ((1-\lambda^2\cos(\theta))^2+\lambda^4\sin^2(\theta)),
\] 
the sign of which is $(-1)^n$ since $\lambda>1$. Thus, $(p_2,p_8)\times S^1$ is $S^1_{(-1)^{n}\eta}$ as a framed manifold.

In a neighborhood of $(p_3,p_9)\times S^1$, we again deform $F$ slightly by collapsing a neighborhood of $(p_3,p_9)$ down to $(p_3,p_9)$ and then translating it slightly along the adjoining interval, just as we did in the proof of \cref{prop:fiberwise Lefschetz cobordism class}, so that there is no longer any fixed point to consider here. 

In a neighborhood of $(p_4,p_{10})\times S^1$. The map $F$ is constant across fibers so we get either the negative or positive standard framing. Since $\Psi^{2}_{S^1}F$ is attracting in two directions and then interchanges the two directions, $(p_4,p_{10})\times S^1$ has the positive standard framing. 

In a neighborhood of $(p_5,p_{11})\times S^1$, we employ the same trick as near $(p_3,p_9)\times S^1$ allowing us to disregard the cobordism class of this manifold. 

In a neighborhood of $(p_6,p_{12})\times S^1$, the map $F$ is again constant across the fibers. So as a framed manifold $(p_6,p_{12})\times S^1$ is $S^1_{\pm\std}$. Observe that $\Psi^{2}_{S^1}F$ repels in the $(n-1)$ directions associated to $p_6$, repels in the $(n-1)$ directions associated to $p_{12}$ and then swaps the two sets of coordinates. To compute the sign of the framing we need to compute the local degree of $\id-\Psi_{S^1}^2F$. Working in local coordinates, $\Psi_{S^1}^2F$ is represented by the $2(n-1)\times 2(n-1)$ block matrix
\[
M=\left[ 
\begin{array}{c|c}
0_{n-1} & \lambda I_{n-1}\\ \hline
\lambda I_{n-1} & 0
\end{array}
\right]
\]
with $\lambda>1$. Then $\det(I-M)=(1-\lambda^2)^{n-1}$, the sign of which is $(-1)^{n-1}$. Therefore, $(p_2,p_3)\times S^1$ is $S^1_{(-1)^{n-1}\std}$ as a framed manifold. 

The claim then follows by recalling that the $C_2$-action on each of these manifolds is $C_2\times S^1$ and that this determines the $C_2$-framing entirely after what we have argued thus far.
\end{proof}

We now wish to show that the above cobordism class is non-trivial. We must first recall the equivariant transfer. 

\begin{df}\label{df:equivariant transfer}
Let $G$ be a finite group and $\rho$ its regular representation. The \textbf{equivariant transfer} is a map 
\[
\pi_n(\Sigma_+^\infty BG)\xrightarrow{\tr} \pi_n^G(\Sigma_+^\infty EG)
\]
defined as follows. Note that, since spheres are compact, any element of $\pi_n(\Sigma_+^\infty BG)$ is determined by its behavior on some finite skeleton of $BG$. The same is true for $EG$. Let $BG^{(l)}$ and $EG^{(l)}$ be such finite skeleta. Choose an equivariant fiberwise embedding $EG^{(l)}\to BG^{(l)}\times k\rho$ with $k$ sufficiently large. Now apply a Pontryagin-Thom collapse map onto the image of $EG^{(l)}$ in $BG^{(l)}\times k\rho$ to obtain a map 
\[
\Sigma^{k\rho}BG^{(l)} \to \Sigma^{k\rho}EG^{(l)}.
\]
Smashing this map with $S^n$ gives a map 
\[
\Sigma^{k\rho+n}BG^{(l)}\to \Sigma^{k\rho+n}EG^{(l)}
\]
with which we may postcompose to obtain
\[
\pi_n(\Sigma_+^\infty BG)\xrightarrow{\tr} \pi_n^G(\Sigma_+^\infty EG).
\]
\end{df}

\begin{lem}\label{lem:orbits transfer}
For $G$ a finite group, the following diagram commutes:
\begin{figure}[H]
\center
\begin{tikzcd}
\omega_n^G(EG) \arrow[r] \arrow[d,"(-)/G"]& \pi_n^G \left(\Sigma_+^\infty EG\right)\\
\omega_n(BG) \arrow[r] & \pi_n\left(\Sigma_+^\infty BG\right)\arrow[u,"\tr"]
\end{tikzcd}
\end{figure}
where both horizontal arrows are Pontryagin-Thom isomorphisms, the left arrow passes to $G$ orbits, and the right arrow is the transfer map of \cref{df:equivariant transfer}.
\end{lem}
\begin{proof}
Let $M\in \omega_n^G(EG)$ be a free $G$-manifold equipped with an $\R^n$-framing. The image of $M$ in $\pi_n^G(\Sigma_+^\infty EG)$ is given by first embedding $M$ into  $S^{k\rho+n+d}$ and then taking the composite
\begin{align*}
S^{kW+n+d}
& \to D(k\rho\oplus\R^{n+d})/S(k\rho\oplus\R^{n+d})\\
& \to D(\nu)/S(\nu)\\
& \to D(M\times (k\rho\oplus\R^d))/S(M\times(k\rho\oplus\R^d))\\
&\to D(EG\times(k\rho\oplus\R^d))/S(EG\times(k\rho\oplus\R^d)).
\end{align*}
In the above, $\nu$ is the normal bundle of $M\to S^{k\rho+n+d}$, the second map is a collapse map, and the third map is a trivialization of the normal bundle induced by the framing. This composite is a map:
\begin{equation}\label{eq:horizontal}
S^{k\rho+n+d}\longrightarrow D(EG\times(k\rho+d))/S(EG\times (k\rho+d))\cong \Sigma_+^{k\rho+d}EG
\end{equation}

The details of this construction can be found in \cite[Definition 2.48]{williams_framed_PT}.

On the other hand, mapping $M$ to $M/G$, then applying the non-equivariant Pontryagin-Thom isomorphism gives
\begin{align*}
S^{n+d}
& \to D(n+d)/S(n+d)\\
& \to D(\nu)/S(\nu)\\
& \to D(M/G\times \R^d)/S(M/G\times \R^d)\\
& \to D(BG\times \R^d)/S(BG\times R^d)\cong \Sigma_+^d BG 
\end{align*}
We then take this element of $\pi_n(\Sigma_+^\infty BG)$ suspend it by $k\rho$, factor through a finite subskeleton of $BG$, and compose with the map of \cref{df:equivariant transfer} to obtain

\begin{equation}\label{eq:orbits-transfer}
(S^{k\rho+n+d}\to \Sigma^{k\rho+d} BG^{(l)} \to \Sigma^{k\rho+d}EG^{(l)})\in\pi_n^G(\Sigma_+^\infty EG)
\end{equation}

as desired. We now wish to show that \cref{eq:horizontal} and \cref{eq:orbits-transfer} define the same element of $\pi_n^G(\Sigma_+^\infty EG)$. 

We may define a transfer map $\Sigma^dM/G\to \Sigma^d M$ analogously to \cref{df:equivariant transfer}. Since $M$ is the pullback of $EG$ along the map $M/G\to BG$, the map 
\[
\Sigma^{k\rho+d} M/G\to \Sigma^{k\rho+d} M\to \Sigma^{k\rho +d} EG^{(l)}
\]
is equivariantly homotopic to  
\[
\Sigma^{k\rho+d} M/G\to \Sigma^{k\rho+d} BG^{(l)}\to \Sigma^{k\rho +d} EG^{(l)}.
\]

Thus, \cref{eq:orbits-transfer} is homotopic to 
\begin{align*}
\Sigma^{k\rho}S^{n+d}
&\to \Sigma^{k\rho}D(M/G\times \R^d)/S(M/G\times\R^{d})\\
&\to \Sigma^{k\rho}D(M\times\R^d)/S(M\times \R^d)\\
&\to \Sigma^{k\rho}D(EG\times\R^d)/S(EG\times\R^d).
\end{align*}
Since the second and third maps in this composition are both Pontryagin-Thom collapse maps, their composition is a collapse map as well. We can then see that \cref{eq:horizontal} and \cref{eq:orbits-transfer} define the same element at the level of the underlying manifolds. It remains to show that they define the same element at the level of framed manifolds.  

Since $M$ is $\R^n$-framed, the trivialization of $\nu(M,k\rho\oplus\R^{n+d})$ in the $k\rho$ coordinates is canonical. Therefore, $D(M\times (k\rho\oplus\R^d))/S(M\times (k\rho\oplus\R^d))$ is equivariantly homeomorphic to $\Sigma^{k\rho}D(M\times \R^d)/S(M\times\R^d)$. Putting this together with the above rearrangement of \cref{eq:orbits-transfer}, shows that \cref{eq:horizontal} and \cref{eq:orbits-transfer} define the same element of $\pi_n^G(\Sigma_+^\infty EG)$ as desired. 
\end{proof}

\begin{prop}\label{prop:nontrivial fuller trace}
The fiberwise Fuller trace
\[
R_{S^1}(\Psi_{S^1}^{2}f)^{C_2}:S^1\times \Sph \to \Sigma_{+S^1}^\infty \Lambda_{S^1}^{\Psi_{S^1}^{2}f}E\times_{S^1}E
\]
of the map of \cref{const:example} is non-trivial.
\end{prop}

\begin{proof}
In \cref{prop:cobordism class of fiberwise Fuller trace}, we showed that the cobordism class associated to $R_{S^1}(\Psi_{S^1}^{2}f)^{C_2}$ corresponds to 
\[
M=C_2\times S^1_{-\std}\amalg C_2\times S^1_{\std}\amalg C_2\times S^1_{(-1)^{n}\eta}\amalg C_2\times S^1_{(-1)^{n-1}\std}
\]
in $\omega_{0}^{C_2}(S^1\to S^1)$. Note that to do this we projected $E_\theta$ to $*$. It suffices to show that $M$ is not null-cobordant because, if $M$ fails to be null-cobordant in $\omega_0^{C_2}(S^1\to S^1)$, then the associated element in $\omega_0^{C_2}(\Lambda_{S^1}^{\Psi_{S^1}^{2}f}E\times_{S^1}E\to S^1)$ will be non-trivial. By \cref{thm:equivariant parameterized PT} this would show that $R_{S^1}(\Psi_{S^1}^{2}f)^{C_2}$ is non-trivial. 

Note that \cref{lem:split section} holds equivariantly by the same argument as above. We may thus, consider the image of $M$ under the equivariant parameterized Pontryagin-Thom isomorphism as an element of $\pi_0^{C_2}(\Sph)\oplus\pi_0^{C_2}(\Sph^{-1})$. Again by the equivariant analog of \cref{lem:split section}, the 3 copies of $C_2\times S^1_{\pm\std}$ correspond to constant sections on $S^1$ and thus, to elements of $\pi_0^{C_2}(\Sph)\oplus \pi_0^{C_2}(\Sph^{-1})$ which are zero in the second coordinate. Therefore, it suffices to show that $C_2\times S^1_\eta$ is non-zero as an element of $\omega_1^{C_2}(*)\cong\pi_0^{C_2}(\Sph^{-1})$.

The tom Dieck splitting identifies $\pi_0^{C_2}(\Sph^{-1})$ with $\pi_1(\Sph)\oplus\pi_1^{C_2}(\Sigma_+^\infty EC_2)$. Since $C_2\times S^1_{\pm\eta}$ is a free $C_2$-manifold, the equivariant Pontryagin-Thom isomorphism maps it entirely into the $\pi_1^{C_2}(\Sigma_+^\infty EC_2)$ summand. Next we will apply \cref{lem:orbits transfer}.

At this point it is convenient to choose point-set models for $EC_2$ and $BC_2$. We will choose $S^\infty$ and $\RP^\infty$ respectively. Taking $C_2\times S^1_{\pm\eta}\in\omega_1^{C_2}(S^\infty)$, we observe that the map $C_2\times S^1_{\pm\eta}\to S^\infty$ is $C_2$-equivariantly null-homotopic. So when we take orbits, the image of the reference map $S^1_{\pm\eta}\to \RP^\infty$ is a single point. After applying the Pontryagin-Thom isomorphism, the resulting formal desuspension of the map $S^{k+1}\to S^k\wedge_+ \RP^\infty$ factors through the formal desuspension of $S^{k+1}\to S^k$. In fact, if we split $\pi_1(\Sigma_+^\infty \RP^\infty)\cong \pi_1(\Sph)\oplus\pi_1(\Sigma^\infty \RP^\infty)$, then the image of $S^1_{\pm\eta}\to *\to BC_2$ under the Pontryagin-Thom isomorphism is $(\pm\eta,0)$ where $\eta$ is a suspension of the Hopf-fibration. Finally, since 
\[
\pi_1(\Sigma_+^\infty BC_2)\xrightarrow{\tr} \pi_1^{C_2}(\Sigma_+^\infty EC_2)
\]
is an isomorphism, the image of $C_2\times S^1_{\pm\eta}$ in $\pi_1^{C_2}(\Sigma_+^\infty EC_2)$ is non-trivial. Then the claim follows.  
\end{proof}

We have now proved all of the pieces of our main theorem.

\begin{thm}\label{thm:theorem c}
There is a family of endomorphisms $f\colon E\to E$ over $S^1$ such that the fiberwise Reidemeister traces of $f$ and $f^2$ are zero but the fiberwise Fuller trace of $f$ is non-zero.
\end{thm}

\begin{proof}
The claim follows immediately from \cref{lem:trivial f}, \cref{prop:trivial f^2}, and \cref{prop:nontrivial fuller trace}.
\end{proof}

\subsection{A manifold example}

While \cref{thm:theorem c} answers \cite[Conjecture 1.9]{mp2} in the affirmative, our construction does not have manifold fibers. We construct a closed manifold example satisfying \cite[Conjecture 1.9]{mp2} in the following corollary to \cref{thm:theorem c}. 

\begin{cor}
There is some family of endomorphisms of closed manifolds $F:E\to E$ parameterized over a base space $B$ for which the fiberwise Reidemeister traces $R_B(F)$ and $R_B(F^2)$, both vanish but for which the fiberwise Fuller trace $R_{B}(\Psi^2(F))^{C_2}$ is non-zero. 
\end{cor}

\begin{proof}
We will base the construction off of \cref{const:example}. Replace $S^n$, $S^{n-1}$, and $I$ in the space $E_\theta$ in \cref{const:example} with $S^n\times D^1$, $S^{n-1}\times D^2$, and $I\times D^n$ respectively. Smoothly glue these together in the same arrangement as the pieces are wedged together in \cref{const:example} to obtain a handlebody $M_\theta$, the core of which is $E_\theta$. Let $M=M_\theta\times S^1$, and let $DM$ be the double of $M$ (which is a closed manifold). Now define a fiberwise endormorphism of $DM$ by collapsing $DM$ onto one copy of $M$, deformation retracting $M$ onto $E$, applying the map of \cref{const:example}, embedding $E$ back into $M$, and finally embedding $M$ back into $DM$. Call this new map $F\colon DM\to DM$. Then the fiberwise Reidemeister and Fuller traces of $F$ are exactly the same as those for the map $f\colon E\to E$ above. 
\end{proof}

\section{Topological Restriction Homology and the Ghost Map}

In this section we prove that \cref{const:example} gives rise to an element of the kernel of the ghost map on $\pi_1(\TR(\Sph))$. We begin with the necessary definitions. 

\begin{df}
The topological restriction homology of a ring spectrum $A$ is defined to be the homotopy limit of the fixed points of $\THH(A)$ along the restriction maps:
\[
\THH(A)^{C_n}\xrightarrow{R} \THH(A)^{C_k}
\]
\end{df}
as described in \cite{bhm}. The particular ring spectra of interest in fixed point theory are of the form $\Sigma_+^\infty\Omega X$ as $\THH(\Sigma_+^\infty \Omega X)\simeq \Sigma_+^\infty \Lambda X$. In \cite{mp2}, Malkiewich and Ponto give a definition of ``twisted topological restriction homology'' inspired by \cite{LM}. In particular, if $X$ is equipped with an endomorphism $f$, then the twisted topological restriction homology of $X$ is 
\[
\TR(X;f) = \holim_{n,R} \left( \Sigma_+^\infty \Lambda^{\Psi^n f} X^{\times n} \right)^{C_n}.
\]

This description then allows us to consider the ``infinite Fuller trace" $R(\Psi^\infty f)^{C_\infty}$ as an element of $\pi_0\TR(X;f)$. 

We will focus on the ring spectrum $\Sph$ as it is the simplest case which provides satisfactory proof of concept. We can obtain elements of $\pi_0^{C_n}(\Sph)$ from the Fuller trace of $f\colon M\to M$ by taking the composition 
\[
\Sph \xrightarrow{R(\Psi^n f)^{C_n}} \left(\Sigma_+^\infty \Lambda^{\Psi^n f} M^{\times n}\right) \to \Sph^{C_n}.
\]

This composite defines a class in $\pi_0(\TR(\Sph))$ as  
\[
\TR(\Sph) \simeq \prod_{n\geq 1} \THH(\Sph)^{C_n} \simeq \prod_{n\geq 1} \Sph^{C_n}.
\]

For the sphere spectrum the ghost map is a map 
\[
\TR(\Sph) \to \prod_{n\geq 1} \THH(\Sph)\simeq \prod_{n\geq 1} \Sph.
\]

Using the tom Dieck splitting, we may identify
\[
\TR(\Sph)\simeq \prod_{n\geq 1} \Sph^{C_n}\simeq\prod_{n\geq 1} \Sph_{hC_n}\simeq  \prod_{n\geq 1} \Sigma_+^\infty BC_n
\] 
Along this identification, the ghost map in the $n^{th}$ coordinate is given by the sum of equivariant transfers $\Sigma_+^\infty BC_k\to \Sph$ for all $k\mid n$.

Now, in order to obtain a class in $\pi_1(\TR(\Sph))$ via fixed point theory, we recall that the fiberwise Fuller trace of the bundle map from  \cref{const:example} is a map 
\[
S^1\times \Sph \xrightarrow{R_{S^1}(\Psi_{S^1}^2 f)^{C_2}} \Sigma_{+S^1}^\infty \lambda_{S^1}^{\Psi_{S^1}^2 f}. E\times_{S^1} E
\]
Then composing with the map induced by $E_\theta \to *$ we obtain 
\[
S^1\times \Sph \xrightarrow{R_{S^1}(\Psi_{S^1}^2 f)^{C_2}} \left(\Sigma_{+S^1}^\infty \lambda_{S^1}^{\Psi_{S^1}^2 f} E\times_{S^1} E\right)^{C_2} \rightarrow \left(\Sigma_{+S^1}^\infty S^1\right)^{C_2}
\]
Therefore the fiberwise Fuller trace gives a class in $\pi_0^{C_2}(\Gamma_S^1 \Sigma_{+S^1}^\infty S^1)$. 

\begin{rmk}
In the language of parameterized spectra, $\Gamma_{S^1}\Sigma_{+S^1}^\infty S^1$ is the same thing as $r_*r^*\Sph$ the spectrum of sections of the pullback of $\Sph$ along $r:S^1\to *$. For the ease of both the reader and the writer, we will denote it as such throughout this section.
\end{rmk}

We will now prove that the example of \cref{const:example} gives rise to an element of the kernel of the ghost map. 

\begin{lem}
Let $\left( r_* r^*\Sph\right)^{C_n}$ be the $C_n$-fixed points of the sections of the pullback of the $C_n$-sphere along $r:S^1\to *$. Then 
\[
\left( r_* r^*\Sph\right)^{C_n} \simeq r_*r^* (\Sph)^{C_n}
\]
\end{lem}
\begin{proof}
First recall from \cref{lem:split section} that $\Gamma_{S^1}\Sigma_{+S^1}^\infty S^1 = r_*r^* \Sph$ is equivalent to $\Sph\vee \Omega\Sph$. The exact same proof shows that 
\[
\left(r_*r^*\Sph\right)^{C_n}\simeq\left( \Sph\vee \Omega\Sph\right)^{C_n}
\]
Moreover, since genuine fixed points commute with products (and thus wedges) and homotopy limits (and thus $\Omega$) we obtain 
\[
\left(r_*r^* \Sph \right)^{C_n}\simeq \Sph^{C_n}\vee \Omega\Sph^{C_n}.
\]
Then reversing the argument of \cref{lem:split section} gives the desired result.
\end{proof}

\begin{prop}\label{commutative diagram}
The following diagram commutes:
\begin{figure}[H]
\begin{tikzcd}
{\prod_{n\geq 1} \pi_0^{C_n}(r_* r^* \Sph)} \arrow[d] \arrow[r] & {\prod_{n\geq 1} \pi_0^{C_n}(\Sph) \bigoplus \prod_{n\geq 1} \pi_1^{C_n}(\Sph)}\arrow[d]\\
\prod_{n\geq 1} \pi_0(r_* r^* \Sph) \arrow[r] & \prod_{n\geq 1} \pi_0(\Sph) \bigoplus \prod_{n\geq 1} \pi_1(\Sph)
\end{tikzcd}
\end{figure}
where the horizontal maps are the splittings of the previous lemma, the right map is given by applying $\pi_0(-)\oplus \pi_1(-)$ to the ghost map, and the left map is given by applying $\pi_0r_*r^*(-)$ to the ghost map (which is well defined by the previous lemma).
\end{prop}

\begin{proof}
First note that since both $r_*$ and $r^*$ commute with fixed points and products the left map is well defined (it is important here that we are working over a compact manifold base). Then by tracing through the proof of \cref{lem:split section} we observe that the splitting $r_*r^*\Sph\simeq \Sph\vee \Sph^{-1}$ is natural. Thus, the diagram commutes.  
\end{proof}

\begin{lem}\label{ghost}
The map
\[
\prod_{n\geq 1} \pi_0^{C_n} (r_* r^* \Sph)\to \prod_{n\geq 1} \pi_0(r_*r^* \Sph)
\]
induced by the ghost map sends $\R_{S^1}(\Psi^n_{S^1} f)^{C_n}\in \pi_0^{C_n}(r_*r^*\Sph)$ to $\prod_{j\mid n} R_{S^1}(f^j)\in \prod_{j\mid n} \pi_0 (r_*r^*\Sph)$.
\end{lem}

\begin{proof}
The proof is an immediate consequence of \cite[Theorem 1.1]{mp2} and \cref{commutative diagram}.
\end{proof}

\begin{prop}\label{any periodicity}
\cref{const:example} is such that $R_{S^1}(\Psi_{S^1}^n f)$ is non-zero for $n$ even and $R_{S^1}(f^n)$ is zero for any $n$.
\end{prop}

\begin{proof}
We first show $R_{S^1}(f^n)$ is zero for any $n$. First note that if $n$ is odd there are no $n$-periodic points so that $R_{S^1}(f^n)$ must be zero. For $n$ even, the exact same analysis as in \cref{reidemeister zero} shows that $R_{S^1}(f^n)$ vanishes. Similarly, the exact same computation as in \cref{fuller nonzero} shows that $R_{S^1}(\Psi_{S^1}f^n)$ is non-zero when $n$ is even. 
\end{proof}

\begin{thm}
\cref{const:example} defines an element in the kernel of the ghost map from $\TR(\Sph)\to \prod \THH(\Sph)$ on $\pi_1$.
\end{thm}

\begin{proof}
By \cref{any periodicity}, $R_{S^1}(\Psi_{S^1}^n f)$ is non-zero whenever $n$ is even. Now by \cref{ghost,any periodicity} we see that $R_{S^1}(\Psi_{S^1}^n f)$ is a nonzero class in the kernel of the map 
\[
\prod_n \pi_0^{C_n}(r_*r^*\Sph) \to \prod_n \pi_0(r_*r^* \Sph).
\] 
Then \cref{commutative diagram} shows that we obtain a non-zero class in the kernel of the ghost map
\[
\prod_n \pi_0^{C_n}(\Sph) \oplus \prod_n \pi_1^{C_n}(\Sph) \to \prod_n \pi_0(\Sph)\oplus \prod_n \pi_1(\Sph).
\]
Since the ghost map is injective the image of $R_{S^1}(\Psi_{S^1}^n f)^{C_n}$ must be zero in $\pi_0(\Sph)$. It must therefore, be non-zero in $\pi_1(\Sph)$ (for $n$ even). We therefore, have produced a class in the kernel of the ghost map 
\[
\TR(\Sph)\simeq \prod_{n} \Sph_{C_n} \to \prod_{n} \Sph \simeq \prod_n \THH(\Sph)
\]
on $\pi_1$. 
\end{proof}

This theorem then gives a proof of concept of the slogan stated in the introduction that the kernel of the ghost map measures the difference between the Fuller trace and the collection of Reidemeister traces of iterates.

\bibliographystyle{amsalpha}
\bibliography{streamlined}

\end{document}